\documentclass[12pt]{amsart}
\usepackage{graphicx}
\usepackage[dvips]{epsfig}
\usepackage{pinlabel}
\usepackage{amsmath}
\usepackage{amsfonts}
\usepackage{latexsym}
\usepackage{amssymb}
\usepackage[usenames,dvipsnames]{xcolor}
\usepackage{amsthm}
\usepackage[all]{xypic}
\usepackage{enumitem}

\input xy
\xyoption{all}

\numberwithin{equation}{section}

\newsavebox{\textvisiblespacebox}
\savebox{\textvisiblespacebox}{\texttt{aa}}
\newcommand\vartextvisiblespace[1][\wd\textvisiblespacebox]{%
  \makebox[#1]{\kern.1em\rule{.4pt}{.3ex}%
  \hrulefill%
  \rule{.4pt}{.3ex}\kern.1em}%
}
\newcommand{\x}{\vartextvisiblespace}

\theoremstyle{remark} \theoremstyle{definition}

\newtheorem{defn}{Definition}[section]
\newtheorem{Rem}[defn]{Remark}

\theoremstyle{plain}
\newtheorem{Thm}[defn]{Theorem}
\newtheorem{Prop}[defn]{Proposition}
\newtheorem{Lem}[defn]{Lemma}

\def\co{\colon\thinspace}

\newcommand{\OP}{\operatorname}

\DeclareMathAlphabet{\mathdj}{U}{msb}{m}{n}

\newcommand{\id}{\operatorname{Id}}
\newcommand{\Hom}{\operatorname{Hom}}

\def\E{\ifmmode{\mathbb E}\else{$\mathbb E$}\fi} 
\def\N{\ifmmode{\mathbb N}\else{$\mathbb N$}\fi} 
\def\R{\ifmmode{\mathbb R}\else{$\mathbb R$}\fi} 
\def\Q{\ifmmode{\mathbb Q}\else{$\mathbb Q$}\fi} 
\def\C{\ifmmode{\mathbb C}\else{$\mathbb C$}\fi} 
\def\H{\ifmmode{\mathbb H}\else{$\mathbb H$}\fi} 
\def\Z{\ifmmode{\mathbb Z}\else{$\mathbb Z$}\fi} 
\def\P{\ifmmode{\mathbb P}\else{$\mathbb P$}\fi} 
\def\T{\ifmmode{\mathbb T}\else{$\mathbb T$}\fi} 
\def\SS{\ifmmode{\mathbb S}\else{$\mathbb S$}\fi} 
\def\DD{\ifmmode{\mathbb D}\else{$\mathbb D$}\fi} 

\begin{document}

\title{Noncommutative augmentation categories}

\author{Baptiste Chantraine}
\author{Georgios Dimitroglou Rizell}
\author{Paolo Ghiggini}
\author{Roman Golovko}

\address{Universit\'e de Nantes, France.}
\email{baptiste.chantraine@univ-nantes.fr}

\address{University of Cambridge, United Kingdom.}
\email{g.dimitroglou@maths.cam.ac.uk}

\address{Universit\'e de Nantes, France.}
\email{paolo.ghiggini@univ-nantes.fr}

\address{Alfr\'{e}d R\'{e}nyi Institute of Mathematics, Hungary.}
\email{golovko.roman@renyi.mta.hu}

\thanks{The first author is partially supported by the
  ANR project COSPIN (ANR-13-JS01-0008-01) and the ERC starting grant
  G\'eodycon.  The second author is supported by the grant KAW
  2013.0321 from the Knut and Alice Wallenberg Foundation. The third
  author is partially supported by the ERC starting grant
  G\'eodycon. The fourth author is supported by the ERC Advanced Grant
  LDTBud.}

\begin{abstract}
To a differential graded algebra with coefficients in a
noncommutative algebra, by dualisation we associate an
$A_\infty$-category whose objects are augmentations. This
generalises the augmentation category of Bourgeois and Chantraine
\cite{augcat} to the noncommutative world.
\end{abstract}


\maketitle

\markboth{Chantraine, Dimitroglou Rizell, Ghiggini,
Golovko}{Noncommutative augmentation categories}

\section{Introduction}
Differential graded algebras (DGAs for short) were introduced by
Cartan in \cite{cartan1954groupes} and occur naturally in a number
of different areas of geometry and topology. We are here interested
in those that appear in the context of Legendrian contact homology,
which is a powerful contact topological invariant due to Chekanov
\cite{Chekanov_DGA_Legendrian} and Eliashberg, Givental and Hofer
\cite{Eliashberg_&_SFT}. In its basic setup, this theory associates
a differential graded algebra, called the {\em
  Chekanov-Eliashberg DGA}, to a given Legendrian submanifold of a
contact manifold. The DGA homotopy type (or even, stable tame
isomorphism type) of the Chekanov-Eliashberg DGA is independent of
the choices made in the construction and invariant under isotopy
through Legendrian submanifolds. Because of some serious analytical
difficulties, Legendrian contact homology has been rigorously
defined only for Legendrian submanifolds of contactisations of
Liouville manifolds \cite{LCHgeneral} and in few other sporadic
cases \cite{Chekanov_DGA_Legendrian, 1024.57014, Sabloff_thesis,
Licata_Sabloff, Ekholm_Ng_subcritical}.

Since the Chekanov-Eliashberg DGA is semifree and fully
noncommutative, it can be difficult to extract invariants from it.
In fact, as an algebra, it is isomorphic to a tensor algebra (and
therefore is typically of infinite rank) and its differential is
nonlinear with respect to the generators.

To circumvent these difficulties, Chekanov introduced his
linearisation procedure in \cite{Chekanov_DGA_Legendrian}: to a
differential graded algebra equipped with an augmentation he
associates a chain complex which is generated, {\em as a module}, by
the generators of the DGA {\em as an algebra}. The differential then
becomes linear at the price of losing the information which is
contained in the multiplicative structure of the DGA, but at least
the homology of the linearised complex is computable.  It is well
known that the set of isomorphism classes of linearised homologies
is invariant under DGA homotopy; see e.g.~\cite[Theorem
2.8]{Bourgeois_Survey}. Thus, linearised Legendrian contact homology
provides us with a computable Legendrian isotopy invariant.

In order to recover at least part of the nonlinear information lost
in the linearisation, one can study products in the linearised
Legendrian contact homology groups induced by the product structure
of the Chekanov-Eliashberg DGA.

Civan, Koprowski, Etnyre, Sabloff and Walker in
\cite{Productstructure} endowed Chekanov's linearised chain complex
with an $A_\infty$-structure.  This construction was generalised in
\cite{augcat} by the first author and Bourgeois, who showed that a
differential graded algebra naturally produces an
$A_\infty$-category whose objects are its augmentations. In
dimension three, the $A_\infty$-category constructed by the first
author and Bourgeois admits a unital refinement defined by Ng,
Rutherford, Shende, Sivek and Zaslow in~\cite{NRSSZ}.  The latter article also establishes an equivalence between this unital $A_\infty$-category and one defined in terms of derived sheaves of microlocal rank one with microsupport given by a fixed Legendrian knot. Our expectation is that the $A_\infty$-structures constructed here correspond to such sheaves being of arbitrary microlocal rank.

$A_\infty$-algebras are by now classical structures which were first
introduced by Stasheff in \cite{HomotopyAssociativity} as a tool in
the study of `group-like' topological spaces. Fukaya was the first
to upgrade the notion of an $A_\infty$-algebra to that of an
$A_\infty$-category. In \cite{MorseHomotopy} he associated an
$A_\infty$-category, which now goes under the name of the
\emph{Fukaya category}, to a symplectic manifold.  See
\cite{Seidel_Fukaya} for a good introduction. Inspired by Fukaya's
work \cite{MorseHomotopy}, Kontsevich in
\cite{kontsevich1995homological} formulated the {\em homological
mirror symmetry conjecture} relating the derived Fukaya category of
a symplectic manifold to the derived category of coherent sheaves on
a ``mirror'' manifold.

The construction in \cite{Productstructure} and \cite{augcat}
defines $A_\infty$-operations only when the coefficient ring of the
DGA is commutative. The goal of this paper is to extend that
construction to noncommutative coefficient rings in the following
two cases:
\begin{enumerate}
\item[(I)] the coefficients of the DGA as well as the augmentations
  are taken in a unital noncommutative algebra, or
\item[(II)] the coefficients of the DGA as well as the augmentations
  are taken in a noncommutative {\em Hermitian algebra}. (See
Definition~\ref{defn: hermitian algebra}.) This case includes both
finite-dimensional algebras over a field and group rings.
\end{enumerate}
 Case (II) is obviously included in Case (I), but we will see that there is a particularly
nice  alternative construction of an $A_\infty$-structure  in case
(II) which gives a different result. We refer to Subsections
\ref{subsec:case1} and \ref{subsec:case2} for the respective
constructions. Both generalisations above are sensible to study when
having Legendrian isotopy invariants in mind, albeit for different
reasons.

Case (I) occurs  because there are  Legendrian submanifolds whose
Chekanov-Eliashberg DGA does not admit augmentations in any unital
algebra of finite rank over a commutative ring, but  admits an
augmentation in a unital noncommutative infinite-dimensional one
(for example, in their characteristic algebras). The first such
examples were Legendrian knots constructed by Sivek in
\cite{TheContHomofLegKNotswithMAXTBI} building on examples found by Shonkwiler and Shea Vela-Vick in \cite{Shink_Vela}.
 From them, the second and fourth authors
constructed higher dimensional examples in
\cite{EstimNumbrReebChordLinReprCharAlg}. Observe that any
differential graded algebra has an augmentation in its
``characteristic algebra'', introduced by Ng in \cite{Ngcomputable},
which is the quotient of the DGA by the two-sided ideal generated by
its boundaries. This algebra is in general noncommutative and
infinite-dimensional, and any augmentation factors through it. It is
of course possible that the characteristic algebra vanishes, but it
does so if and only if the DGA is acyclic \cite{Characteristic}. The
\emph{complex} that we will define in case (I) (but not the higher
order operations) was used in \cite{Caps} by the second author in
order to deduce that a Legendrian submanifold with a non-acyclic
Chekanov-Eliashberg DGA does not admit a displaceable Lagrangian
cap.

Finally, we note that the construction we give in Case (I) is
closely related to the $A_\infty$-structures and bounding cochains
with noncommutative coefficients as introduced by Cho, Hong and
Sui-Cheong in their recent work \cite{Cho}.  Namely, the (uncurved)
$A_\infty$-structures that we produce from a DGA and its
augmentations can be seen to coincide with the (uncurved)
$A_\infty$-structures produced by their bounding cochains.

Case (II) also occurs naturally in the context of Legendrian contact
homology. For example, in \cite{Satellites} Ng and Rutherford show
that augmentations of certain satellites of Legendrian knots induce
augmentations in matrix algebras for the Chekanov-Eliashberg DGA of
the underlying knot. Moreover, coefficients in a group ring appear
naturally if one considers the Chekanov-Eliashberg DGA with
coefficients ``twisted'' by the fundamental group of the Legendrian
submanifold.  We learned this construction from Eriksson-\"{O}stman,
who makes use of it in his upcoming work \cite{Albin}. This version
of Legendrian contact homology  can be seen as a natural
generalisation of Morse homology and Floer homology with
coefficients twisted by the fundamental group; see the work
\cite{KFloer} by Sullivan and \cite{Damian_Lifted} by Damian. In the
setting of Legendrian contact homology with twisted coefficients, an
exact Lagrangian filling gives rise to an augmentation taking values
in the group ring of the fundamental group of the filling. See the
work \cite{Cthulhu} by the authors for more details, were Legendrian
contact homology  with twisted coefficients is used to  study the
topology of Lagrangian fillings and cobordisms.

In Section \ref{sec:computation} we outline how our construction can be used as an efficient computational tool for distinguishing a Legendrian knot from its Legendrian mirror in the case when there are no augmentations in a commutative algebra. Note that it, in general, it is much easier to extract invariants from the $A_\infty$-algebra compared to the DGA.

Finally, we recall that Legendrian contact homology is not the only
place where noncommutative graded algebras appear in symplectic
geometry. Another source is cluster homology, a proposed
generalisation of Lagrangian Floer homology due to Cornea and
Lalonde \cite{Cluster}, which is supposed to provide an alternative
approach to the $A_\infty$-structures in Floer homology introduced
by Fukaya, Oh, Ohta and Ono \cite{fooo}.

\subsection*{Acknowledgements}
We would like to thank the organisers of the twenty-second G\"okova
Geometry-Topology conference for the wonderful mathematical, as well
as natural, environment,  and the Institut Mittag-Leffler for
hospitality during the program ``Symplectic Geometry and Topology'',
when part of this article was written. The first author has also
benefited from the hospitality of CIRGET in Montr\'eal and IAS in
Princeton. Last but not least, we thank Lenny Ng  and Stiven Sivek
for useful discussions. The example in Section~\ref{sec: Lenny's
example}, suggested to us by Ng, was a major source of inspiration.

\section{Algebraic preliminaries}
\label{sec:preliminaries}

In this section, we fix some notations and recall some basic
definitions from the theory of modules over (possibly
noncommutative) algebras. For more details of this theory, we refer
to \cite{dummit2004abstract}. We also introduce some notation that
will simplify the various formulas for the $A_\infty$-structures
that we will define.

\subsection{Bimodules and tensor products}
In this paper $R$ will always denote a commutative ring and $A$ will
denote  a unital algebra over $R$ which is not necessarily
commutative. Important examples will be the matrix algebra
$M_{n}(R)$ corresponding to the endomorphisms of the free $R$-module
$R^n$, and the group ring $R[G]$ of an arbitrary group $G$.

For $R$-modules $M$, $N$ we denote by
\[ M \otimes N :=  M \otimes_R N \]
their tensor product as $R$-modules. Moreover, if $M$, $N$ are
$A$--$A$-bimodules, their (balanced) tensor product is denoted by
\[ M \boxtimes N := M\otimes_A N.\]
We recall that the balanced tensor product is the quotient of $M
 \otimes N$  by $ma\otimes n = m\otimes an$ for all $a \in A$,
$m\in M$ and $n\in N$.

 We also remind that a free $A$--$A$-bimodule $M$ on generating
  set $B$ is an $A$--$A$-bimodule $M$ and a map $i:B\to M$  of sets such
  that, for any $A$--$A$-bimodule $N$ and any map $f:B\to N$  of sets,
  there is a unique $A$--$A$-bimodule morphism $\overline{f}:M\to
  N$ such that $\overline{f}\circ i=f$. The elements of ${\mathcal B}= i(B)$
in $M$  are a basis for $M$. The free $A$--$A$-bimodule with basis
${\mathcal B}$ will often be identified with $\oplus_{c \in
{\mathcal B}} (A \otimes_{R} A)$, where the action of $A$ from the
left (resp. right)  acts by multiplication from the left (resp.
right) on the left  (resp. right) factor. Elements of $M$ will also
written as linear combinations of elements of the form $ a_+ c a_-$
with $a_\pm \in A$ and $c \in {\mathcal B}$.

A \emph{grading} of an $A$--$A$-bimodule $M$ in the group $\Z/\Z2
\mu$ is a direct sum decomposition $M = \oplus_{g \in \Z/\Z2 \mu}
M_g$. If $m \in M_g$, we write $|m| =g$. The tensor product of
graded bimodules is graded by the usual rule.

\subsection{Tensor algebras}
Given an $A$--$A$-bimodule $M$, we define the {\em tensor algebra}
of $M$ as the algebra
\[ \mathcal{T}_A(M):=\bigoplus_{n=0}^\infty M^{\boxtimes n}\]
with the multiplication
\[ \mathfrak{m} \colon \mathcal{T}_A(M) \boxtimes \mathcal{T}_A(M) \to
\mathcal{T}_A(M), \]
\[ \mathfrak{m}((m_1 \boxtimes \ldots \boxtimes m_i)
\boxtimes (n_1 \boxtimes \ldots \boxtimes n_j)) = m_1 \boxtimes
\ldots \boxtimes m_i \boxtimes n_1 \boxtimes \ldots \boxtimes n_j.
\]
 Here we have used the notation
\begin{align*}
M^{\boxtimes 0} &:= A,\\
M^{\boxtimes n} & := \underbrace{M \boxtimes \hdots \boxtimes M}_n.
\end{align*}
We will call $M^{\boxtimes 0}$ the {\em zero-length part} of
$\mathcal{T}_A(M)$.

If $M$ is a graded bimodule, then the tensor algebra
$\mathcal{T}_A(M)$ inherits a grading
\[\mathcal{T}_A(M) = \bigoplus_{g \in \Z/\Z2 \mu} \mathcal{T}_A(M)_g \]
by requiring that the zero-length part lives in degree zero; i.e.
$M^{\boxtimes 0} \subset \mathcal{T}_A(M)_g$ and
\[\mathfrak{m} \colon \mathcal{T}_A(M)_{g_1} \boxtimes \mathcal{T}_A(M)_{g_1} \to
\mathcal{T}_A(M)_{g_1+g_2}.\]

In this article, algebra maps will always be unital. Algebra maps
$\mathcal{T}_A(M) \to \mathcal{T}_A(N)$ between tensor algebras over
$A$ will always be morphisms of $A$--$A$-bimodules, and in
particular they will restrict to the identity $A=M^{\boxtimes 0} \to
N^{\boxtimes 0}=A$ on the zero-length parts. On the other hand,
algebra maps $\mathcal{T}_A(M) \to B$ for a general $R$-algebra $B$
will be $R$--$R$-bimodule morphisms, and their restriction to
$M^{\boxtimes 0}=A$  induces a unital $R$-algebra morphisms $A \to
B$. Algebra maps defined on $\mathcal{T}_A(M)$ are determined by
their restrictions to $M^{\boxtimes 0}=A$ and $M^{\boxtimes 1} = M$.

Note that, as in the case when $A$ is commutative, there is a notion
of ``free product'' of tensor algebras defined by
\[ \mathcal{T}_A(M) \star \mathcal{T}_A(N):=\mathcal{T}_A(M \oplus N),\]
which is again a tensor algebra. Moreover, for algebra maps $f_i
\colon \mathcal{A}_i \to \mathcal{B}_i$, $i=1,\hdots,n$, between
tensor algebras $\mathcal{A}_1, \ldots \mathcal{A}_n, \mathcal{B}_1,
\ldots \mathcal{B}_n$, there is a naturally induced algebra map
\[f_1 \star \hdots \star f_n \colon \mathcal{A}_1 \star \hdots \star \mathcal{A}_n
\to \mathcal{B}_1 \star \hdots \star \mathcal{B}_n\] between the
corresponding free products.

In all our applications, $M$ will be a free $A$--$A$-bimodules. In
this case, if the elements $c_1, \ldots, c_m,\ldots$ freely generate
the $A$--$A$-bimodule $M$, they also generate the algebra
$\mathcal{T}_A(M)$ in the following sense: every element in
$M^{\boxtimes n}$ can be written as $a_0c_{i_1}a_1 \ldots
c_{i_n}a_n$, with $a_0, \ldots, a_n \in A$.
\subsection{Duals}
\label{sec:bimoduledual} For an $R$-module $M$, we denote by $M^*$
the dual module $\Hom_R(M,R)$. If $M$ is free with a given finite
basis ${\mathcal B}$, then $M^*$ is again free with a dual basis
${\mathcal B}^*$ induced by ${\mathcal B}$. We recall that, for any
$c \in {\mathcal B}$, the dual basis element $c^*$ is the element of
$M^*$ which maps $c$ to $1 \in R$ and any other element of
${\mathcal B}$ to $0$. Hence, when the basis is part of the data, we
will identify $M$ with $M^*$ by identifying $c$ with $c^*$ for all
$c \in {\mathcal B}$.
 If ${\mathcal B}$ is not finite, the above construction
only gives an injection $M \to M^*$.

Given a $R$-module map $f:M\rightarrow N$, we denote the adjoint
morphism by $f^*:N^*\rightarrow M^*$. Again, if $M$ and $N$ are free
with given finite bases, then we denote $f^* \colon N\rightarrow M$.

Let $A$ be any $R$-algebra. We can regard $A$ as a nonfree
$A$--$A$-bimodule over itself. For an $A$--$A$-bimodule $M$ we will
define $M^\vee := \Hom_{A-A}(M,A)$ in the sense of bimodules.
Observe that in general $M^\vee$ only has the structure of an
$R$-module. If $M$ is a free and finitely generated
$A$--$A$-bimodule with a preferred basis ${\mathcal B}$, then
$M^\vee$ can be identified with a free $A$-module with the same
basis. The correspondence is given by
\begin{equation}\label{pippo}
M^\vee \ni \varphi \mapsto \sum \limits_{c \in {\mathcal B}}
\varphi(c)c.
\end{equation}
Again, any morphism $f \colon M \to N$ of bimodules gives rise to an
adjoint morphism
\[ f^\vee \colon N^\vee \to M^\vee \]
which, typically, is only a morphism of bimodules.

We define morphisms of $R$-modules $\psi_n \colon (M^\vee)^{\otimes
n} \to (M^{\boxtimes n})^\vee$ by
\begin{equation}\label{pluto}
\psi_n(\beta_1 \otimes \hdots \otimes \beta_n)(m_1 \boxtimes \hdots
\boxtimes m_n)= \beta_1(m_1) \hdots \beta_n(m_m)
\end{equation}
for $\beta_i \in M^\vee$ and $m_i \in M$. Note that $\psi_n(\beta_1
\otimes \hdots \otimes \beta_n )$ is well defined on the balanced
tensor product since the $\beta_i$ are bimodule morphisms. The maps
$\psi_n$ cannot be seen as morphisms of bimodules in any sensible
way.

If $M$ is graded, then the dual modules $M^*=\Hom_R(M,R)$ and
$M^\vee =\Hom_{A-A}(M,A)$ (when defined) are also graded with
gradings $(M^\vee)_g:=(M_{g-1})^\vee$ and $(M^*)_g:=(M_{g-1})^*$,
i.e.~the \emph{suspension} of the dual gradings.

\subsection{Hermitian algebras}
We are interested in duals $\Hom_R(M,R)$ of $A$--$A$-bimodules $M$
for algebras $A$ which are not necessarily finitely generated free
$R$-modules. For that reason, in order to have a better behaving
theory, we will introduce some additional structure on the algebra
$A$.

A commutative ring $R$ is an {\em involutive ring} if it is endowed
with an involution $r \mapsto \overline{r}$ (called {\em
conjugation}), which is also a ring homomorphism. The prototypical
example to keep in mind is the field of complex numbers, but we will
also allow involutive rings where the conjugation is the identity.
From now on every ring will be tacitly considered involutive,
possibly with a trivial involution.

\begin{defn}\label{defn: hermitian algebra}
A {\em Hermitian algebra} $(A, \star, \mathfrak{t})$ over an
involutive ring $R$ consists of:
\begin{itemize}
\item an $R$-algebra $A$,
\item a map
\begin{gather*}
\star \colon A \to A, \\
a \mapsto a^\star,
\end{gather*}
satisfying
\begin{enumerate}
\item $(r a + s b)^\star = \overline{r}a^\star + \overline{s} b^\star$ for all $r,s \in R$
and all $a,b \in A$,
\item $(ab)^\star=b^\star a^\star$, and
\item $(a^\star)^\star=a$, and
\end{enumerate}
\item a Hermitian form
$$\mathfrak{t} \colon A \times A \to R$$
such that
$$\mathfrak{t}(ba, c)=\mathfrak{t}(a, b^\star c) = \mathfrak{t}(b, ca^\star)$$
for all $a,b,c \in A$, and which is non-degenerate in the following
strong sense. For any $n \ge 1$, the morphism
\begin{gather*}
\iota \colon \underbrace{A \otimes_R \hdots \otimes_R A}_n \to
(\underbrace{A \otimes_R \hdots \otimes_R A}_n)^*,\\
x \mapsto \iota_x,
\end{gather*}
determined by
\[\iota_{a_1 \otimes \hdots \otimes a_n}(a_1' \otimes \hdots \otimes a_n') =
\mathfrak{t}(a_1',a_1)\cdot \hdots \cdot \mathfrak{t}(a_n',a_n) \in
R \] is injective.
\end{itemize}
\end{defn}
From (2) and (3) it follows that $1^\star = 1$. In fact $a =
(a^\star)^\star = (1 a^\star)^\star = a 1^\star$, and similarly $a=
1^\star a$ for all $a \in A$.
\begin{Rem}
There are two cases in which the above non-degeneracy for $n > 1$
follows from the case $n=1$:
\begin{enumerate}
\item $A$ is free (possibly infinitely generated) as an $R$-module and
$\mathfrak{t}$ is induced by the canonical pairing of its basis
elements, or
\item $R$ is a domain.
\end{enumerate}
\end{Rem}
\noindent Note that, if the conjugation on $R$ is trivial,
$\mathfrak{t}$ is a symmetric bilinear form.

Our main examples of Hermitian algebras will be the group ring
$R[G]$ over an arbitrary group $G$ and the matrix algebras
$M_{n}(R)$; in both cases $R$ is an arbitrary commutative ring. On
the group ring the involution is induced by the inverse in $G$, i.e.
$g^\star = g^{-1}$ on the basis elements $g \in G$, and
$\mathfrak{t}$ is the scalar product for which the group elements $g
\in G$ form an orthonormal basis. On the matrix algebra $M_{n}(R)$
we distinguish whether the conjugation on $R$ is trivial or not. In
the first case, the involution in $M_{n}(R)$ is the transposition,
and in the second case it is the adjoint (i.e. the transposition
followed by the conjugations). In both types of matrix algebras,
$$\mathfrak{t}(a,b):=\OP{tr}(b^\star a), \:\: a,b \in M_{n}(R)$$
is given by the trace. For simplicity, from now on we will consider
only Hermitian algebras over commutative rings whose conjugation is
trivial.

\subsection{Bimodules over Hermitian algebras and their duals}

Let $M$ be an arbitrary $A$--$A$-bimodule over a Hermitian algebra
$A$. The involution on $A$ allows us to define an  $A$--$A$-bimodule
structure on $M^*=\Hom_R(M,R)$ by
\begin{equation}\label{bimodule via star}
(a_1\varphi a_2)(m) := \varphi(a_1^\star m a_2^\star)
\end{equation}
for any $a_1, a_2 \in A$, $m \in M$ and $\varphi \in \Hom_R(M,R)$.
\begin{Lem}\label{adjoint is a bimodule map}
Let $A$ be a Hermitian algebra and let $f \colon M \to N$ be a
morphism between $A$--$A$-bimodules $M$ and $N$. Then the adjoint
map $f^* \colon N^* \to M^*$ is also a morphism of
$A$--$A$-bimodules.
\end{Lem}
\begin{proof}
The proof is a simple computation: Let $\varphi \in N^*$, $m \in M$
and $a_+, a_- \in A$. Then
\[f^*(a_+ \varphi a_- ) (m)= \varphi (a_+^\star f(m) a_-^\star)= \varphi (f(a_+^\star
m a_-^\star))= a_+ f^*(\varphi(m)) a_-. \]
\end{proof}

Now consider a free $A$--$A$-bimodule $M$ with a preferred basis
${\mathcal B}$. The bilinear pairing $\mathfrak{t}$ on $A$ and the
basis ${\mathcal B}$ induce an $R$-bilinear pairing on each
$A$--$A$-bimodule $M^{\boxtimes n}$, $n \ge 0$, which, on elements
of the form $a_0d_1a_1\ldots d_na_n$ with $a_i \in A$ and $d_i \in
{\mathcal B}$ is defined by
\[ \langle a_0d_1a_1\ldots d_na_n, a_0'd_1'a_1'\ldots d_n'a_n' \rangle =
\begin{cases}
\mathfrak{t}(a_0, a_0')\cdot \ldots \cdot \mathfrak{t}(a_n, a_n') &
\text{if }
d_i=d_i', \: i=1,\hdots,n, \\
0 & \text{otherwise},
\end{cases} \]
where $d_i' \in {\mathcal B}$ as well. It can be checked explicitly
that
\[\langle a_1ma_2,n \rangle = \langle m,a_1^\star na_2^\star \rangle \]
holds for all $m,n \in M^{\boxtimes n}$ and $a_1,a_2 \in A$. By the
assumption of nondegeneracy of $\mathfrak{t}$, this pairing then
induces an injection
\begin{gather}
\iota^{(n)} \colon M^{\boxtimes n} \hookrightarrow (M^{\boxtimes n})^*,\nonumber\\
m \mapsto \langle m, \x \rangle \in \Hom_R(M^{\boxtimes n},R),
\label{inclusion in dual}
\end{gather}
for each $n \ge 0$. The identifications $\iota^{(n)}$ also satisfy
the following property.
\begin{Lem}
Let $A$ be a Hermitian algebra and let $M$ be a free
$A$--$A$-bimodule with a preferred basis. Then the inclusion
$\iota^{(n)}$ for $n \ge 1$ is a morphism of $A$--$A$-bimodules for
the bimodule structure on $(M^{\boxtimes n})^*$ described in
Equation~\eqref{bimodule via star}.
\end{Lem}
\begin{proof}
For any $m \in M$ and $a_1,a_2 \in A$ we compute
\[a_1\langle m, \x \rangle a_2 = \langle m,a_1^\star \x a_2^\star \rangle =
\langle a_1ma_2,\x \rangle,\] which shows the claim.
\end{proof}
We will often tacitly identify $M^{\boxtimes n}$ with its image in
$(M^{\boxtimes n})^*$ under the inclusion $\iota^{(n)}$. In general,
it is not necessarily the case that an $R$-module morphism $f \colon
M \to N$ between free infinitely generated $R$-modules has an
adjoint morphism $f^* \colon N^* \to M^*$ that restricts to a
morphism $f^* \colon N \to M \subset M^*$ for these submodules $N
\subset N^*$ and $M \subset M^*$. However, this turns out to be the
case for a large class of maps that we are interested in here. First
we give the following useful formula.

\begin{Lem}
\label{lem:mainformula} Let $A$ be a Hermitian algebra, and let
$M,N$ be free finitely generated $A$--$A$-bimodules with preferred
bases. Consider a morphism $g \colon M \to N^{\boxtimes n}$, $n \ge
1$, of $A$--$A$-bimodules which vanishes on all basis elements
except a single $c \in M$, on which it takes the form
\[  g(c)= a_0d_1a_1 \ldots d_n a_n,\]
where $d_i \in N$ again are basis elements and $a_i \in A$. Then
\begin{eqnarray*}
\lefteqn{g^*( \langle a_0'd_1'a_1'\ldots a_{n-1}'d_n'a_n', \x \rangle)=}\\
&=& \begin{cases} \langle (\mathfrak{t}(a_1,a_1')\cdot\hdots\cdot
\mathfrak{t}(a_{n-1},a_{n-1}')) \cdot a_0' a_0^\star \cdot c \cdot
a_n^\star  a_n',\x \rangle, & \text{if } d_i=d_i', \:
i=1,\hdots,n,\\
0, & \text{otherwise}
\end{cases}
\end{eqnarray*}
for arbitrary basis elements $d_i' \in N$ and $a_i' \in A$.
\end{Lem}

\begin{proof}
This is a simple verification; for a generator $d$ we have
$$ g^*( \langle a_0'd_1'a_1'\ldots a_{n-1}'d_n'a_n', \x \rangle)(d)=\langle a_0'd_1'a_1' \ldots a_{n-1}'d_n'a_n', g(d) \rangle.$$
This can be nonzero only if $ d=c$ where we get:
\begin{eqnarray*}
\lefteqn{\langle a_0'd_1'a_1'\ldots a_{n-1}'d_n'a_n', g(c) \rangle=}\\
&=&
\begin{cases}
\big(\mathfrak{t}(a_0,a_0')\cdot\mathfrak{t}(a_1,a_1')\cdots
\mathfrak{t}(a_n,a_n')
\big)& \text{if } d_i=d_i', \: i=1,\hdots,n,\\
0, & \text{otherwise.}
  \end{cases}
\end{eqnarray*}
On the other hand the expression
$$\langle (\mathfrak{t}(a_1,a_1')\cdot\hdots\cdot \mathfrak{t}(a_{n-1},a_{n-1}'))\cdot a_0' a_0^\star \cdot c \cdot a_n^\star  a_n',d \rangle$$
is nonzero only when $ d=c$. Moreover,
$$\langle (\mathfrak{t}(a_1,a_1')\cdot\hdots\cdot \mathfrak{t}(a_{n-1},a_{n-1}'))
\cdot a_0' a_0^\star \cdot c \cdot a_n^\star a_n', c
\rangle=\mathfrak{t}(a_0,a_0')\cdot\mathfrak{t}(a_1,a_1')\cdots
\mathfrak{t}(a_n,a_n')$$ holds since
$\mathfrak{t}(a,b)=\mathfrak{t}(a b^\star,1)$. This concludes the
proof.
\end{proof}
We similarly compute the following relation.

\begin{Lem}
\label{lem:mainformula2} Let $A$ be a Hermitian algebra, and let
$M,N$ be free finitely generated $A$--$A$-bimodules with preferred
bases. Consider a morphism $g \colon M \to N^{\boxtimes n}$, $n \ge
1$, of $A$--$A$-bimodules of as in Lemma \ref{lem:mainformula}. For
the morphism
\[ G=\id_M^{\boxtimes k} \boxtimes g \boxtimes \id_M^{\boxtimes l} \colon
M^{\boxtimes (k+1+l)} \to M^{\boxtimes k} \boxtimes N^{\boxtimes n}
\boxtimes M^{\boxtimes l}\] we then have
\begin{eqnarray*}
\lefteqn{G^*( \langle x \boxtimes a_0'd_1'a_1'\ldots
a_{n-1}'d_n'a_n' \boxtimes y,
\x \rangle)=}\\
&=& \begin{cases} \langle (\mathfrak{t}(a_1,a_1')\cdot\hdots\cdot
\mathfrak{t}(a_{n-1},a_{n-1}')) \cdot x \boxtimes a_0' a_0^\star
\cdot c \cdot a_n^\star  a_n' \boxtimes y,\x \rangle, &
\text{if } d_i=d_i', \: i=1,\hdots,n,\\
0, & \text{otherwise}
\end{cases}
\end{eqnarray*}
for arbitrary basis elements $d_i' \in N$ and $a_i' \in A$, and any
$x \in M^{\boxtimes k}$, $y \in M^{\boxtimes l}$.
\end{Lem}
\begin{proof}
Again it is a matter of checking that
$(\mathfrak{t}(a_1,a_1')\cdot\hdots\cdot
\mathfrak{t}(a_{n-1},a_{n-1}'))\cdot x \boxtimes a_0' a_0^\star
\cdot c \cdot a_n^\star a_n' \boxtimes y$ represents the dual $G^*(
\langle x \boxtimes a_0'd_1'a_1'\ldots a_{n-1}'d_n'a_n' \boxtimes y,
\x \rangle)$. This follows by a straight-forward computation
similarly to the proof of Lemma \ref{lem:mainformula}.
\end{proof}

\begin{Prop}
\label{prp:mainprop} Let $A$ be a Hermitian algebra and let $M,N$ be
free finitely generated $A$--$A$-bimodules with preferred bases.
Given a morphism $f \colon M \to N^{\boxtimes n}$ of bimodules where
$n>0$, for any $k,l\geq 0$ we consider the induced bimodule morphism
\[ F=\id_M^{\boxtimes k} \boxtimes f \boxtimes \id_M^{\boxtimes l} \colon
M^{\boxtimes (k+1+l)} \to M^{\boxtimes k} \boxtimes N^{\boxtimes n}
\boxtimes M^{\boxtimes l}.\] Then the adjoint
\[F^* \colon ( M^{\boxtimes k} \boxtimes N^{\boxtimes n} \boxtimes M^{\boxtimes l} )^*
\to (M^{\boxtimes (k+l+1)})^*\] is a morphism of $A$--$A$-bimodules
restricting to a morphism of the form
\begin{gather*}
F^* \colon M^{\boxtimes k} \boxtimes N^{\boxtimes n} \boxtimes M^{\boxtimes l} \to M^{\boxtimes (k+l+1)},\\
F^* =\id_M^{\boxtimes k} \boxtimes f^* \boxtimes \id_M^{\boxtimes
l},
\end{gather*}
on the submodules defined by the inclusion in
Equation~\eqref{inclusion in dual}.
\end{Prop}
\begin{proof}
  The fact that $F^{\ast}$ is a morphism of $A$-$A$-bimodules follows
  from Lemma~\ref{adjoint is a bimodule map}. The latter statement
  follows directly from Lemma~\ref{lem:mainformula2}. Namely, the
  morphism $F$ considered here can be written as a finite sum of
  morphisms $\id_M^{\boxtimes k} \boxtimes g \boxtimes
  \id_M^{\boxtimes l}$ satisfying the assumptions of Lemma~\ref{lem:mainformula2}.

\end{proof}
\begin{Rem}
Note that in order for property (2) above to hold, it is crucial
that $n>0$. For instance, the property is not satisfied for the
adjoint $m^*$ of the multiplication $m: R[G] \otimes R[G] \to R[G]$
for the group ring of an infinite group.
\end{Rem}

\section{Differential graded algebras over noncommutative
rings} In this section we recall some facts about differential
graded algebras which are well known for commutative coefficient
rings.

\subsection{Definitions}\label{sec:definitions}
Let $R$ be a unital commutative ring and $A$ a (not necessarily
commutative) unital algebra over $R$.
\begin{defn}
  A {\em differential graded algebra} $({\mathcal A}, \partial)$ over
  $A$ is a unital $\Z / 2 \mu \Z$-graded algebra $\mathcal{A}$ over $A$ whose
  differential $\partial \colon {\mathcal A} \to {\mathcal A}$
  is a morphism of $A$--$A$-bimodules satisfying the following
  properties:
\begin{enumerate}
\item $\partial \circ \partial =0$,
\item $\partial$ has degree $-1$, and
\item $\partial(xy)= \partial(x)y + (-1)^{|x|} x \partial(y)$ for all homogeneous elements
$x,y \in {\mathcal A}$. \label{leibniz}
\end{enumerate}
\end{defn}
The last equality is known as \emph{graded Leibniz rule}, and tells
us that $\partial$ is a derivation. Above $|x| \in \Z/2 \mu \Z$ is
the degree of $x$. The graded Leibniz rule (and the fact that $1$ is
homogeneous of degree $0$) implies that $\partial (1)=0$. In fact,
$\partial (1)= \partial(1 \cdot 1)=\partial(1) \cdot 1 + 1 \cdot
\partial(1)= \partial(1) +
\partial(1)$. Since $\partial$ is a morphism of $A$--$A$--bimodules, this implies
that $\partial(a \cdot 1)=0$ for all $a \in A$.

In this article we will consider only ``semifree'' differential
graded algebras with finitely many generators, i.e.~whose underlying
algebra is the tensor algebra
$${\mathcal A}= \mathcal{T}_A(M)=\bigoplus_{n=0}^\infty M^{\boxtimes n}$$
where $M$ is a finitely generated graded free $A$--$A$-bimodule.
Moreover we will always assume that $M$ comes with a specified
finite basis ${\mathcal B} = \{c_1, \ldots, c_k \}$ over $A$
consisting of homogeneous elements.

The differential is determined by its action on $M$, where it
decomposes as
\begin{equation}\label{decomposition of differential}
\partial|_M = \partial_0 + \partial_1 + \partial_2 + \ldots,
\end{equation}
where $\partial_n \colon M \to M^{\boxtimes n}$.   Clearly
$\partial_i =0$ for $i$ sufficiently large because $M$ is finitely
generated. We refer to $\partial_0\colon M = M^{\boxtimes 1}
\rightarrow  M^{\boxtimes 0} = A$ as the \textit{constant part} of
$\partial$. From the Leibniz rule it follows that the differential
of an element $x= a_0c_{i_1}a_1 \ldots c_{i_n}a_n \in M^{\boxtimes
n}$ is
\begin{align}
 \partial
x&=\sum_{j=0}^{n-1}(-1)^{|c_{i_1}|+\cdots+|c_{i_{j-1}}|}a_0c_{i_1}a_1\cdots
a_{j-1}
\partial(c_{i_j})a_j\cdots c_{i_n}a_n \nonumber \\
&=\sum_{k=0}^\infty\sum_{j=0}^{n-1}(-1)^{|c_{i_1}|+\cdots+|c_{i_{j-1}}|}a_0c_{i_1}a_1
\cdots a_{j-1} \partial_k(c_{i_j})a_j\cdots c_{i_n}a_n. \label{eq:2}
 \end{align}
Combining this with the fact that $\partial\circ\partial=0$, we get
the following relation for the maps $\partial_i$:
\begin{equation}
\label{eq:leibnizpartial} \sum_{k+l-1=n \atop k>0, \: l \ge 0 }
\sum_{i=0}^{k-1} (\sigma^{\boxtimes i}\boxtimes\partial_l\boxtimes
\id^{\boxtimes (k-1-i)})\circ \partial_k=0,
\end{equation}
for any fixed $n \ge 0$. Here $\sigma$ is the automorphism of $M$
which maps a homogenous element $m$ to $(-1)^{|m|}m$, and
$\sigma^0:=1$.

\subsection{Changing the coefficients}
\label{sec:change} Assume that we are given a differential graded
algebra with coefficients in $A$ and a morphism $A \to B$ of unital
$R$-algebras. In certain situations it will be useful to consider a
change of coefficients from $A$ to $B$. We recall that $B$ has an
induced structure of an $A$--$A$-bimodule and that $M_B:=B \boxtimes
M \boxtimes B$ is a free $B$--$B$-bimodule for any free
$A$--$A$-bimodule $M$.
\begin{Lem}\label{change of ring}
  Let $(\mathcal{A}, \partial)$ be a semi-free differential graded
  algebra over $A$ such that ${\mathcal A}$ is isomorphic to
  $\mathcal{T}_A(M)$ as an algebra, and let $f: A\rightarrow B$ be a unital algebra
  morphism. Then there exist:
\begin{itemize}
\item a unique semi-free differential graded algebra $(\mathcal{A}_B, \partial_B)$
over $B$ such that  $\mathcal{A}_B$ is isomorphic  to
$\mathcal{T}_B( M_B)$ as an algebra, and
\item a unique morphism $\hat{f} \colon \mathcal{A} \to \mathcal{A}_B$ of
  unital graded algebras
\end{itemize}
satisfying the following properties:
\begin{enumerate}
\item  $\hat{f}$ is the natural morphism of unital $R$-algebras defined
uniquely by the requirements that it restricts to $f$ on
$M^{\boxtimes 0} = A$, and induces a graded bijection between the
generators of $\mathcal{A}$ and $\mathcal{A}_B$, and
\item $\hat{f}\circ \partial = \partial_B \circ \hat{f}$,
\end{enumerate}
i.e.~$\hat{f}$ is a unital DGA morphism.
\end{Lem}
\begin{proof}

The existence of the algebra morphism $\hat{f}$ is immediate. The
differential $\partial_B$ on $\mathcal{A}_B$ is defined on the image
$\hat{f}(m) \in M_B$ of $m \in M$ to take the value
\[ \partial_B( \hat{f}(m))=\hat{f}(\partial(m)).\] Since $\hat{f}$ is
surjective on the generators of $\mathcal{A}$ this determines
$\partial_B$ uniquely after extending $\partial_B$ using the graded
Leibniz rule. Using the Leibniz rule, the fact that $\hat{f}$ is an
algebra morphism implies that
\[\hat{f}\circ \partial = \partial_B \circ \hat{f}\]
is satisfied on all of $\mathcal{A}$.

It remains to check that $\partial_B^2=0$. Since we clearly have
\[\partial_B^2 \circ \hat{f}=\hat{f}\circ \partial^2=0,\]
the fact that $\hat{f}$ is a bijection on the generators implies
that $\partial_B^2=0$ is satisfied on all of $\mathcal{A}_B$.
\end{proof}
The following changes of coefficients was used in the construction
of the augmentation category in \cite{augcat}, and will also be
relevant in this article. Consider the unital $R$-algebra $A_n:=A
\otimes_R (\bigoplus_{i=1}^n R e_i)$, where $\bigoplus_{i=1}^n R
e_i$ has the ring structure induced by termwise multiplication,
i.e.~$e_i\cdot e_j=\delta_{ij}e_i$ and $\sum e_i=1$. The morphism $f
\colon A \to A_n$ will be the canonical morphism induced by the
above tensor product with $\bigoplus_{i=1}^n R e_i$, i.e.~for which
$f(1) = e_1+e_2+\hdots + e_n$.

\subsection{Augmentations and linearisations}
\label{ssec: augmentations and linearisations} The graded Leibniz
rule is invariant under conjugation by degree-preserving unital
algebra automorphisms. More precisely, let $\phi:
\mathcal{A}\rightarrow\mathcal{A}$ be such an automorphism; then
$\partial_\phi:=\phi^{-1} \circ \partial \circ \phi$ is a
differential and $(\mathcal{A},\partial_\phi)$ is again a
differential graded algebra. We denote by $\Pi_0$ the projection of
$\mathcal{A}$ to the zero length part $M^{\boxtimes 0} =A$. The
constant part of $\partial_\phi$ is given by $\Pi_0 \circ \phi^{-1}
\circ \partial \circ \phi$. In particular, if the map
$\varepsilon:=\Pi_0 \circ \phi^{-1}$ satisfies
$\varepsilon\circ\partial =0$, this constant term vanishes. This
motivates the definition of an augmentation:
\begin{defn}
  Let $B$ be a unital $R$-algebra together with a unital algebra
  morphism $f \colon A \to B$ of $R$-modules. An \emph{augmentation}
  of $({\mathcal A}, \partial)$ into $B$ is a unital DGA morphism
  $\varepsilon \colon {\mathcal A} \to B$ for which
  $\varepsilon|_{M^{\boxtimes 0}}=f$. Here $B$ is regarded as a differential graded
  algebra with trivial differential and concentrated in degree zero. (Therefore
$\varepsilon \circ \partial =0$.)
\end{defn} \noindent
Note that, in particular, $\varepsilon$ is a a morphism of
$A$--$A$-bimodules for the $A$--$A$-bimodule structure on $B$
induced by $f$.

In \cite{Chekanov_DGA_Legendrian}, Chekanov described a
linearisation procedure which uses an augmentation to produce a
differential on the graded $A$-module $M$. While Chekanov originally
defined linearisation for differential graded algebra over a
commutative ring, it is known that his construction works equally
well for differential graded algebras with noncommutative
coefficients. We now recall this construction.

From a differential graded algebra $({\mathcal A}, \partial)$
together with an augmentation $\varepsilon \colon {\mathcal A} \to
B$ we produce a new differential graded algebra as follows. By
applying Lemma~\ref{change of ring} to $\varepsilon|_{M^{\boxtimes
0}}=f \colon A \to B$ we obtain a differential graded algebra
$({\mathcal A}_B,\partial_B)$ with coefficients in $B$, and using
the unital DGA morphism $\hat{f} \colon {\mathcal A} \to {\mathcal
A}_B$ we define an augmentation $\varepsilon_B \colon {\mathcal A}_B
\to B$ by the requirement that $\varepsilon_B \circ
\hat{f}=\varepsilon$ holds on the generators. Using $\varepsilon_B$
we define a unital algebra automorphism $\Phi_\varepsilon \colon
{\mathcal A}_B \to {\mathcal A}_B$  determined by
\[ \Phi_\varepsilon(m)=m+\varepsilon_B(m), \quad m \in M_B.\]
We obtain a differential via the conjugation
\[\partial^\varepsilon := \Phi_\varepsilon \circ \partial_B \circ \Phi_\varepsilon^{-1}.\]
Let $\Pi_0 \colon {\mathcal A}_B \to M_B^{\boxtimes 0} = B$ be the
natural projection; then it follows that $\Pi_0 \circ
\partial_B^\varepsilon = \varepsilon_B \circ \partial_B =0$. The
differential graded algebra $(\mathcal{A}_B,\partial^\varepsilon)$
will be said to be obtained from $(\mathcal{A},\partial)$ by
\emph{developing with
  respect to the augmentation $\varepsilon$}.

The fact that $(\partial_B^\varepsilon)_0=0$ will be important in
the next section.  Using this, Equation \eqref{eq:leibnizpartial}
can be rewritten as
\begin{equation}
\label{eq:leibnizaugmented} \sum_{k+l-1=n \atop k,l>0 }
\sum_{i=0}^{k-1}
 (\sigma^{\boxtimes i}\boxtimes(\partial_B^\varepsilon)_l\boxtimes \id^{\boxtimes (k-1-i)})\circ (\partial_B^\varepsilon)_k=0,
\end{equation}
for any fixed $n>0$.

\subsection{The free n-copy DGA}
\label{sec:semisimple}

Let $(\mathcal{A},\partial)$ be a differential graded algebra. We
consider algebras $\mathcal{A}_{ij} = \mathcal{A}$ for $0 \le i,j
\le n$ and form a differential graded algebra $(\mathfrak{A}_n,
\mathfrak{d})$ where $\mathfrak{A}_n$ is, as a graded algebra, the
free product
\[ \mathfrak{A}_n:=\bigstar_{0 \le i ,j \le n}\mathcal{A}_{ij},\]
and the differential $\mathfrak{d}$ is induced by $\partial$ as
follows. If $c$ is a generator of ${\mathcal A}$, we denote by
$c^{ij}$ the generator in ${\mathfrak A}$ which corresponds to the
copy of $c$ in ${\mathcal A}_{ij}$. Then
\begin{itemize}
\item $\mathfrak{d}_0(c^{ij})= \partial_0(c) \in A$ if $i=j$, and $\mathfrak{d}_0(c^{ij})= 0$ if $i \ne j$, and
\item  the coefficient of $ a_0d_1^{i_1j_1}a_1d_2^{i_2j_2} \ldots d^{i_nj_n}_na_n$ in $\mathfrak{d}(c^{ij})$ is equal to the coefficient of\\ $a_0d_1a_1d_2 \ldots d_na_n$ in $\partial(c)$, where $d_i$ is a sequence of generators, given that $i_1=i$, $j_n=j$, and $j_{k-1} = i_k$ are satisfied, while this coefficient otherwise vanishes.
\end{itemize}
As usual, we extend $\mathfrak{d}$ to the whole algebra
$\mathfrak{A}_n$ via the graded Leibniz rule. (Since we have not
proved yet that $\mathfrak{d}^2=0$, strictly speaking,
$(\mathfrak{A}_n, \mathfrak{d})$ is only a graded algebra with a
derivation so far.) A generator $c^{ij}$ will be called \emph{mixed}
if $i \neq j$ and \emph{pure} if $i=j$. Observe that $\mathfrak{d}$
preserves the filtration of $\mathfrak{A}_n$ given by the
$R$-submodules spanned by those words containing at least a number
$m \ge 0$ of mixed generators.

A word $a_0d_1^{i_1j_1}a_1d_2^{i_2j_2} \ldots d^{i_nj_n}_na_n$ will
be called {\em composable} if $j_{k-1} = i_k$ for $2 \le k \le n$.
Words of length zero and one are automatically composable. We define
$\mathfrak{A}_n^c \subset \mathfrak{A}_n$ as the
sub-$A$--$A$-bimodule generated by composable words. It is immediate
to verify that $\mathfrak{d}$ restricts to an endomorphism of
$\mathfrak{A}_n^c$.

To prove that $\mathfrak{d}^2=0$ we use an alternative definition
using the change of coefficients. Recall the ring $A_n=A \otimes
(\bigoplus_{i=1}^n Re_i)$ from the end of Section \ref{sec:change}.
For a semi-free differential graded algebra $(\mathcal{A},\partial)$
we denote by $ ( {\mathcal A}_{A_n}, \partial_{A_n})$ the
differential graded algebra obtained by the change of coefficients
from $A$ to $A_n$ using Lemma \ref{change of ring}.

Recall that there is grading  preserving bijection between the sets
of generators of the respective algebras  $\mathcal{A}$ and $
{\mathcal A}_{A_n}$.

For a generator $c$ of $\mathcal{A}$ we denote $e_ice_j \in
 {\mathcal A}_{A_n}$ by $c^{ij}$. Note that the $c^{ij}$'s generate $ {\mathcal A}_{A_n}$,
albeit not freely. The differential $ \partial_{A_n}$ can now be
expressed as follows on any generator $c$. Given that a term
$a_0d_1a_1d_2 \ldots d_na_n$ appears in $\partial c$, the sum
$$\sum_{i_1, j_n} \sum_{j_{k-1}=i_k}a_0e_{i_1}d_1e_{j_1}a_1d_2^{i_2j_2}
\cdots d_{n-1}^{i_{n-1}j_{n-1}}a_{n-1}e_{i_n}d_n e_{j_n} a_n$$
appears in the expression $\partial_{A_n} c$. (Recall here that the
$e_i$ are in the centre of $A$.) This means that the sum
$$\sum_{j_{k-1}=i_k}e_{i_1}a_0d_1^{i_1j_1}a_1d_2^{i_2j_2}\cdots
d_{n-1}^{i_{n-1}j_{n-1}}a_{n-1}d_n^{i_nj_n}e_{j_n}a_n$$ appears in
the expression $\partial_{A_n}(c^{i_1j_n})$. More generally, we have
$\partial_{A_n}(c^{i_1j_n}) \subset e_{i_1} {\mathcal
A}_{A_n}e_{j_n}$. Since $e_ie_j=\delta_{ij}e_i$, in particular it
follows that $\partial_{A_n}c^{ij}$ has no constant term for
$i\not=j$.
\begin{Lem}
We have $\mathfrak{d}^2=0$.
\end{Lem}
\begin{proof}
The first observation is that it is enough to prove that
$\mathfrak{d}^2(c^{ij})=0$ for every generator $c^{ij}$ of
$\mathfrak{A}_n$. There is an algebra morphism $\pi \colon
\mathfrak{A}_n \to  {\mathcal A}_{A_n}$ such that $\pi(c^{ij})= e_i
c e_j$ and $\pi(1)= e_1 + \ldots + e_n$. It is easy to check that
$\pi \circ \mathfrak{d} =  \partial_{A_n} \circ \pi$, and moreover
$\pi$ is injective on the sub-bimodule $\mathfrak{A}_n^c$ generated
by composable words. Since $\mathfrak{d}^2(c^{ij}) \in
\mathfrak{A}_n^c$ for every generator $c^{ij}$, from $
\partial_{A_n}^2=0$ it follows that $\mathfrak{d}^2=0$.
\end{proof}

From a sequence
$\boldsymbol{\varepsilon}=(\varepsilon_0,\ldots,\varepsilon_n)$ of
augmentations $\varepsilon_i \colon {\mathcal A} \to A$ we define an
algebra morphism $\mathfrak{e} \colon \mathfrak{A}_n \to A$ such
that, on a generator $c^{ij}$ of ${\mathcal A}_{ij}$,
\[\mathfrak{e} (c^{ij})= \begin{cases} { \varepsilon_i(c)} & \text{if } i=j, \text{ and }\\
0 & \text{if } i \ne j.\end{cases}\]
\begin{Lem}
The algebra morphism $\mathfrak{e} \colon \mathfrak{A}_n \to A$ is
an augmentation.
\end{Lem}
\begin{proof}
It suffices to check that $\mathfrak{e}(\mathfrak{d}(c^{ij}))=0$
holds on the generators. By construction, we have
\[\mathfrak{e} (\mathfrak{d}(c^{ij}))= \begin{cases}
\varepsilon_i(\partial(c)) & \text{if } i = j, \text{ and } \\
0 & \text{if  } i \neq j, \end{cases}\] which establishes the claim.
\end{proof}

We can use the augmentation $\mathfrak{e}$ to produce a differential
graded algebra $(\mathfrak{A}_n, \mathfrak{d}^{\mathfrak{e}})$ whose
differential has vanishing constant term by applying the procedure
described in the previous subsection.

\section{$A_\infty$-operations}
\subsection{Case I: coefficients in a general noncommutative
algebra} \label{subsec:case1} Let $(\mathcal{A},\partial)$ be a
differential graded algebra with coefficients in a noncommutative
algebra $A$ over a commutative ring $R$. We further assume that
${\mathcal A} =\mathcal{T}_A(M)$ is a tensor algebra over a free
$A$--$A$-bimodule $M$ with a preferred basis $\{c_1,\hdots, c_k\}$.

Recall that we decompose $\partial|_{M} = \partial_0 + \partial_1 +
\hdots $, where $\partial_n \colon M \to M^{\boxtimes n}$. We start
without the assumption that $\partial_0=0$. There are induced
adjoints $(\partial_n)^\vee \colon (M^{\boxtimes n})^\vee \to
M^\vee$, and we define
\begin{equation} \label{definition of mu_n}
\mu_n := (\partial_n)^{\vee} \circ \psi_n \colon (M^\vee)^{\otimes
n} \to M^\vee, \quad n \ge 1,
\end{equation}
where $\psi_n$ is as defined in Section \ref{sec:bimoduledual}
above. See Diagram \eqref{mubar1}.
\begin{gather}
\xymatrix{ M^\vee & (M^{\boxtimes n})^\vee \ar[l]_{(\partial_n)^\vee} \\
M^\vee \ar@{=}[u]^{\psi_1=\id_{M^\vee}} & \ar[l]^{\mu_n}
\ar[u]_{\psi_n} (M^\vee)^{\otimes n} }\label{mubar1}
\end{gather}
Given elements $m_1, \ldots, m_n \in M^\vee$, we will write
$\mu_n(m_1, \ldots, m_n)$ and $\mu_n(m_1 \otimes \ldots \otimes
m_n)$ interchangeably.

The operations $\mu_n$ can be expressed more concretely as follows.
\begin{Lem}\label{explicit computation I}
If, for every element $c_i$ in the basis of $M$,
\[\partial_nc_i=\sum_I  \sum_{j=1}^{m_{i,I}}a^{i,I}_{j,0}c_{i_1}a^{i,I}_{j,1}
\ldots c_{i_n}a^{i,I}_{j,n}\] with $a^{i,I}_{j,l} \in A$, and
$I=(i_1,\hdots,i_n)$ denoting a multi-index with $1 \le i_l \le k$,
then
\[ \mu_n(b_1c_{i_1},\hdots,b_nc_{i_n})=\sum_{i=1}^k \sum_{j=1}^{m_{i,I}}
(a^{i,I}_{j,0}b_1a^{i,I}_{j,1}\ldots
a^{i,I}_{j,n-1}b_na^{i,I}_{j,n})\cdot c_i\] for each $n \ge 1$ and
any elements $b_i \in A$.
\end{Lem}
\begin{proof}
From Equations \eqref{pippo} and \eqref{pluto} it follows that
\begin{align*}
&\psi_n(b_1c_{i_1}\otimes \hdots \otimes b_nc_{i_n}) (a_0 c_{j_1}a_1
\ldots a_{n-1} c_{j_n} a_n )=\\&=\begin{cases} a_0 b_1 a_1 \ldots
a_{n-1}b_n a_n, & \text{if } c_{i_1}=
c_{j_1}, \ldots, c_{i_n}=c_{j_n}, \\
0, & \text{otherwise}.
\end{cases}
\end{align*}
Then from
$$\mu_n(b_1c_{i_1},\hdots,b_nc_{i_n})(c_i)=
\psi_n(b_1c_{i_1}\otimes \hdots \otimes b_nc_{i_n})(\partial_n
c_i)$$ for all basis elements $c_i$, and from Equation \eqref{pippo}
again, the lemma follows.
\end{proof}
\begin{Rem} The maps $\mu_n$ are $R$-multilinear, but do not satisfy any form
of $A$-linearity in general.
\end{Rem}

The maps $(\mu_i)_{i\geq 1}$ do not necessarily satisfy the
$A_\infty$ relations because the curvature term
$\mu_0:=(\partial_0)^\vee \colon A \to M^\vee$ might be
non-vanishing. Given augmentations of $\mathcal{A}$, this can be
amended.

To an augmentation $\varepsilon \colon {\mathcal A} \to A$ we
associate the element
\[\varepsilon^\vee:=\varepsilon(c_1)c_1+\hdots+\varepsilon(c_k)c_k \in M^\vee,\]
i.e.~the adjoint of $\varepsilon|_M$.
\begin{Rem}
We should think of $\varepsilon^\vee$ as giving rise to a ``bounding
cochain'' in the sense of \cite{fooo} via the infinite sum
\[ \sum_{i=0}^\infty (\varepsilon^\vee)^{\otimes i} \]
living in the completion $\prod_{i=0}^\infty (M^\vee)^{\otimes i}$
of $\bigoplus_{i=0}^\infty (M^\vee)^{\otimes i}$, where
$(\varepsilon^\vee)^{\otimes 0}:=1 \in R$.
\end{Rem}

\begin{defn}\label{defn: A_infty operations case I}
 For a sequence
$\boldsymbol{\varepsilon}=(\varepsilon_0,\varepsilon_1,\hdots,\varepsilon_n)$
of augmentations $ \epsilon_i \colon {\mathcal A} \to A$, we define
the operations
\[\mu^{\boldsymbol{\varepsilon}}_n \colon (M^\vee)^{\otimes n} \to M^\vee\]
via the formulas
\begin{eqnarray*}
\lefteqn{\mu^{\boldsymbol{\varepsilon}}_n(m_1, \ldots, m_n) =}\\
&=&\sum_{i=1}^\infty \sum_{i_0+\hdots+i_n+n=i \atop i_j \ge 0}
\mu_i( (\varepsilon_0^\vee)^{\otimes i_0} \otimes m_1 \otimes
(\varepsilon_1^\vee)^{\otimes i_1} \otimes \hdots\otimes
(\varepsilon_{n-1}^\vee)^{\otimes i_{n-1}} \otimes m_n\otimes
(\varepsilon_{n}^\vee)^{\otimes i_{n}}),
\end{eqnarray*}
where $m_i \in M^{\vee}$  (recall that $(\varepsilon_0^\vee)^{
\otimes 0}=1$ is considered as an element in $R$).
\end{defn}
See Figure \ref{fig:curve} for a geometric explanation of the terms
appearing in $\mu^{\boldsymbol{\varepsilon}}_n(m_1, \ldots, m_n)$.

Given a $n$-tuple of augmentation
$\boldsymbol{\varepsilon}=(\varepsilon_1,\ldots,\varepsilon_n)$ we
denote by $\boldsymbol{\varepsilon}_{ij}=(\varepsilon_i, \ldots,
\varepsilon_j)$ and
$\hat{\boldsymbol{\varepsilon}}_{ij}=(\varepsilon_1, \ldots,
\varepsilon_i, \varepsilon_j, \ldots, \varepsilon_n)$.
\begin{figure}[ht]
\labellist \pinlabel $d_0$ at 100 122 \pinlabel $d_1$ at 10 -6
\pinlabel $d_2$ at 55 -6 \pinlabel $d_3$ at 101 -6 \pinlabel $d_4$
at 145 -6 \pinlabel $d_5$ at 190 -6 \pinlabel $a_0$ at 33 65
\pinlabel $a_1$ at 32 14 \pinlabel $a_2$ at 78 14 \pinlabel $a_3$ at
122 14 \pinlabel $a_4$ at 168 14 \pinlabel $a_5$ at 167 65
\endlabellist
  \centering
  \includegraphics{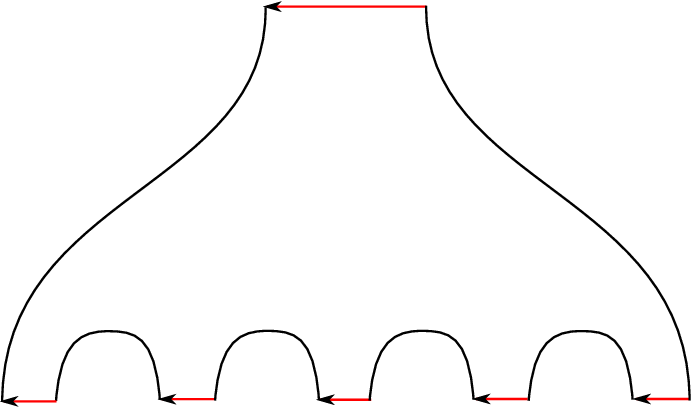}
  \vspace{5mm}
  \caption{The ``pseudoholomorphic disk with punctures'' shown here is
supposed to give a contribution of $+1$ to the coefficient in front
of $a_0d_1a_1d_2a_2d_3a_3d_4a_4d_5a_5$ in the expression
$\partial_5(d_0)$. If the generators $d_1,\hdots,d_5$ are all
distinct, this pseudoholomorphic disc gives a contribution of
$a_0\varepsilon_0(d_1)a_1xa_2\varepsilon_1(d_3)
a_3ya_4\varepsilon_2(d_4)a_5 \in A$ to the coefficient in front of
$d_0$ in the expression of
$\mu^{(\varepsilon_0,\varepsilon_1,\varepsilon_2)}_2(xd_2,yd_4)$.}
  \label{fig:curve}
\end{figure}

\begin{Thm}
\label{thm:case1} For any $n \ge 1$ and fixed sequence
$\boldsymbol{\varepsilon}=(\varepsilon_0,\varepsilon_1,\hdots,\varepsilon_n)$
of augmentations, the above $R$-module morphisms satisfy the
relations
\begin{eqnarray*}
\lefteqn{0=\sum_{n=k-1+l \atop k,l > 0}\sum_{i=1}^k (-1)^\dagger
\mu^{\hat{\boldsymbol{\varepsilon}}_{i-1,i+l-1}}_k
(m_1,\hdots,m_{i-1},}\\
& &
\mu^{\boldsymbol{\varepsilon}_{i-1,i+l-1}}_l(m_i,\hdots,m_{i+l-1}),
m_{i+l}, \hdots,m_n),
\end{eqnarray*}
where $\dagger=|m_1|+\cdots+|m_{i-1}|+(i-1)$, and $m_1,\hdots,m_n
\in M^\vee$ are homogeneous elements. In other words, the above
operations form the morphisms and higher operations of an
$A_\infty$-category over $R$ whose objects consist of the
augmentations $\varepsilon \colon \mathcal{A} \to A$.
\end{Thm}
\begin{proof}
First we handle the case when all augmentations are equal and
trivial (i.e.~sending all generators to zero). In this case the
$A_\infty$-relations readily follow from Formula
\eqref{eq:leibnizaugmented}, which is satisfied when $\partial_0=0$.
Also, see the following diagram.
\begin{gather}
\xymatrix@C=10pc{ (M^{\boxtimes k})^\vee  & (M^{\boxtimes
(k+l)})^\vee \ar[l]_{(\sum \sigma\boxtimes
\partial_l\boxtimes \id)^\vee} \\
(M^\vee)^{\otimes k} \ar[u]^{\psi_k} & \ar[l]^{\sum \sigma\otimes
\mu_l\otimes \id} \ar[u]_{\psi_{k+l}} (M^\vee)^{\otimes (k+l)}.
}\label{mubar2}
\end{gather}

The general case can now be reduced to the above case in the
following manner. First, given a sequence
$\boldsymbol{\varepsilon}=(\varepsilon,\hdots,\varepsilon)$
consisting of a single augmentation, we compute that
$(\mu^{\boldsymbol{\varepsilon}}_i)_{i\geq 1}$ associated to
$(\mathcal{A},\partial)$ are equal to the morphisms $(\mu_i)_{i\geq
1}$ associated to the DGA $(\mathcal{A},\partial^\varepsilon)$
defined in Subsection \ref{ssec: augmentations and linearisations}.
This case thus follows from the above.

Finally, for an arbitrary sequence
$\boldsymbol{\varepsilon}=(\varepsilon_0,\hdots, \varepsilon_n)$ of
augmentations we apply the construction in Section
\ref{sec:semisimple}. Namely, we produce the auxiliary
``semisimple'' differential graded algebra
$(\mathfrak{A}_n,\mathfrak{d})$ and the auxiliary augmentation
$\mathfrak{e} \colon \mathfrak{A}_n \to A$ induced by
$\boldsymbol{\varepsilon}$. Using the notation in Section
\ref{sec:semisimple}, it can be seen that
\[ \mu_n^{\boldsymbol{\varepsilon}}(a_1d_1,\ldots, a_nd_n)= \mu_n^{\mathfrak{e}}
(a_1d_1^{i_1j_1},\ldots, a_nd_n^{i_nj_n}) \] (after identifying the
output with an element of $\mathcal{A}_{i_1j_n} = \mathcal{A}$),
where $i_k < j_k$ and $j_k=i_{k+1}$ holds for all indices $k$, and
$d_i$ is a sequence of basis elements. We have thus managed to
reduce the general case to the first case.

For the sign $\dagger$ it suffices to notice that $(\sigma)^\vee= -
\sigma$ according to sign the convention of Section
\ref{sec:bimoduledual}.
\end{proof}

\subsection{Case II: coefficients in a Hermitian algebra}
\label{subsec:case2} Let $(\mathcal{A}, \partial)$ be a differential
graded algebra with coefficients in  a noncommutative algebra $A$
over a commutative ring $R$. As in the previous subsection, we
assume that  ${\mathcal A}=\mathcal{T}_A(M)$ is a  tensor algebra
over an free $A$--$A$-bimodule $M$ with a preferred basis
$\{c_1,\hdots, c_k\}$. In this subsection we make the assumption
that $A$ is a Hermitian algebra. Recall that there are induced
inclusions $\iota^{(n)} \colon M^{\boxtimes n} \to (M^{\boxtimes
n})^*$ for each $n\ge 0$ induced by the basis on $M$ and by the
bilinear form $\mathfrak{t}$ on $A$.

We define the $A$--$A$-bimodule morphisms
\[ \mu_n := (\partial_n)^* \colon (M^{\boxtimes n})^* \to (M)^*\] for
each $n \ge 1$. In view of Proposition \ref{prp:mainprop} these
morphisms restrict to morphisms $\mu_n := (\partial_n)^* \colon
M^{\boxtimes n} \to M$ under the above inclusions. However, since
$\partial_0$ is not assumed to be zero, these operations might not
give rise to an $A_\infty$ structure in the strict sense. We now
proceed to amend this.

Using Lemma \ref{lem:mainformula}, the operations $\mu_n$ can in
this case be expressed more concretely as follows.
\begin{Lem}\label{explicit computation II}
If, for every element $c_i$ in the basis of $M$,
\[\partial_nc_i=\sum_I  \sum_{j=1}^{m_{i,I}}a^{i,I}_{j,0}c_{i_1}a^{i,I}_{j,1}
\ldots c_{i_n}a^{i,I}_{j,n}\] with $a^{i,I}_{j,l} \in A$, and
$I=(i_1,\hdots,i_n)$ denoting a multi-index with $1 \le i_l \le k$,
then
\begin{eqnarray*}
\lefteqn{ \mu_n(b_0c_{i_1}b_1\hdots b_{n-1}c_{i_n}b_n)=}\\
&=& \sum_{i=1}^k \sum_{j=1}^{m_{i,I}} \langle c_{i_1} b_1\hdots
b_{n-1}c_{i_n}, c_{i_1}a^{i,I}_{j,1} \ldots a^{i,I}_{j,n-1} c_{i_n}
\rangle b_0(a^{i,I}_{j,0})^\star c_i(a^{i,I}_{j,n})^\star b_n
\end{eqnarray*}
for each $n \ge 1$ and any elements $b_i \in A$.
\end{Lem}

Given an augmentation $\varepsilon \colon \mathcal{A} \to A$, we
define the adjoints
\[\varepsilon_{(n)}^* \co A^* \to (M^{\boxtimes n})^*\]
for each $n \ge 0$, where $\varepsilon_{(0)}^*=\id_{A^*}$. Again
these maps are related to the notion of a ``bounding cochain''. As a
side remark, We note that
\begin{equation} \label{poneron}
\varepsilon_{(i)}^*(a)=a \cdot
\varepsilon_{(i)}^*(1)=\varepsilon_{(i)}^*(1) \cdot a
\end{equation}
holds for the $A$--$A$-bimodule structure defined by (\ref{bimodule
via star}).

\begin{Rem}
 When the algebra $A$ is free as an $R$-module and the
pairing $\mathfrak{t}$ is induced by an orthonormal basis, the
``bounding cochains''
\[\varepsilon_{(n)}^* \co A \to (M^{\boxtimes n})^*\]
can be expressed as
\begin{equation}\label{epsilon star explicit}
\varepsilon^*_{(n)}(a)  = \langle a, \varepsilon(\x) \rangle =
 \sum_{a_0 d_1 a_1 \ldots a_{n-1}d_n a_n} \mathfrak{t}(a,\varepsilon(a_0 d_1 a_1 \ldots a_{n-1}
d_n a_n))a_0 d_1 a_1 \ldots a_{n-1}d_n a_n,
\end{equation}
where the sum is taken over all words $a_0 d_1 a_1 \ldots a_{n-1}d_n
a_n$ such that $a_1, \ldots, a_n$ are elements of the orthonormal
basis of $A$ and $d_1, \ldots, d_n$ are elements of the prescribed
basis of $M$ (both allowing repetitions).

As a double check we verify Equation \ref{poneron} for $n=1$ using
Equation \eqref{epsilon star explicit}. If $\{ c_1, \ldots, c_k \}$
is the basis of $M$, we denote $\varepsilon_i = \varepsilon(c_i)$.
Then Equation \eqref{epsilon star explicit} for $n=1$ can be
rewritten as
$$\varepsilon^*_{(1)}(a) = \sum_{i=1}^k \sum_{a_+, a_-} \mathfrak{t}(a, a_+
\varepsilon_i a_-) a_+ c_i a_-,$$ where $a_+$ and $a_-$ run through
the orthonormal basis of $A$. Now we observe that
$$\sum_{a_+} \mathfrak{t}(a, a_+ \varepsilon_i a_-) a_+ =
\sum_{a_+} \mathfrak{t}(a a_-^\star \varepsilon^\star_i, a_+)a_+ = a
a_-^\star \varepsilon^\star_i$$ and therefore we can rewrite
$$\varepsilon^*_{(1)}(a) = \sum_{i=1}^k \sum_{a_-} a a_-^\star \varepsilon^\star_i c_i
a_-.$$ On the other hand we have
\begin{align*}
&a \varepsilon_{(1)}(1) = \sum_{i=1}^k \sum_{a_+, a_-}
\mathfrak{t}(1, a_+ \varepsilon_i a_-) aa_+c_ia_- =\\&= \sum_{i=1}^k
\sum_{a_+, a_-} \mathfrak{t} (a_-^\star \varepsilon_i^\star, a_+)
aa_+c_ia_- = \sum_{i=1}^k \sum_{a_-} aa_-^\star \varepsilon_i^\star
c_i a_-.
\end{align*}
Then half of Equation \eqref{poneron} is verified. The other half is
similar.
\end{Rem}
\begin{defn}\label{defn: A-infty operation case II}
Given a sequence
$\boldsymbol{\varepsilon}=(\varepsilon_0,\varepsilon_1,\hdots,
\varepsilon_n)$ of augmentations of ${\mathcal A}$, we define the
operations
\[\mu^{\boldsymbol{\varepsilon}}_n \colon (M^{\boxtimes n}) \to M^*, \quad n \ge 1,\]
via the formulas
\begin{eqnarray}
\lefteqn{\mu^{\boldsymbol{\varepsilon}}_n(a_0m_1a_1 \ldots
a_{n-1}m_na_n) =}
\\ \nonumber &=& \sum \limits_{i=1}^\infty \sum_{i_0+\hdots +i_n+n=i \atop i_j \ge 0}
\mu_i( (\varepsilon_0)_{(i_0)}^*(a_0)\boxtimes m_1 \boxtimes
(\varepsilon_1)_{(i_1)}^*(m_1) \boxtimes \ldots \boxtimes m_n
\boxtimes (\varepsilon_n)_{(i_{n})}^*(a_{n})),\label{ainfinity}
\end{eqnarray}
where $m_1 \ldots, m_n \in M$ and $a_0, \ldots, a_n \in A$ and
$a_0m_1a_1 \ldots a_{n-1}m_na_n \in M^{\boxtimes n}$ is identified
to an element in $(M^{\boxtimes n})^*$ by the inclusion $\iota_{(n)}
\colon M^{\boxtimes n} \to (M^{\boxtimes n})^*$ (see Equation
\eqref{inclusion in dual}).
\end{defn}
\begin{Lem}
The compositions in Formula \eqref{ainfinity} give rise to a
well-defined map
$$\mu^{\boldsymbol{\varepsilon}}_n \colon M^{\boxtimes n} \to M \subset M^*.$$
\end{Lem}
\begin{proof}
Since
$$\mu^{\boldsymbol{\varepsilon}}_n :=\sum \limits_{i=1}^\infty \sum_{i_0+\hdots +i_n+n=i \atop i_j > 0}\big(\big((\varepsilon_0)^{\boxtimes i_0} \boxtimes \id_M
\boxtimes (\varepsilon_1)^{\boxtimes i_1} \boxtimes \hdots
\boxtimes\id_M \boxtimes (\varepsilon_n)^{\boxtimes i_n} \big)\circ
\partial_i \big)^*,$$ the statement follows from Proposition
\ref{prp:mainprop}.
\end{proof}
\begin{Rem}
The operations $\mu^{\boldsymbol{\varepsilon}}_n$ are morphisms of
$A$--$A$-bimodules by Lemma \ref{adjoint is a bimodule map} and
Equation \eqref{poneron}.
\end{Rem}
The main result of this section is that these operations define an
$A_\infty$-category.
\begin{Thm}\label{thm:caseII}
For any $n \ge 1$ and fixed sequence $\boldsymbol{\varepsilon}=
(\varepsilon_0,\varepsilon_1,\hdots,\varepsilon_n)$ of
augmentations, the operations in Definition \ref{defn: A-infty
operation case II} satisfy the following $A_\infty$ relations:
\begin{equation*}
\sum_{n=k-1+l \atop k,l > 0}\sum_{i=1}^k (-1)^\dagger
\mu^{\hat{\boldsymbol{\varepsilon}}_{i-1,i+l-1}}_k(m_1,\hdots,m_{i-1},
\mu^{\boldsymbol{\varepsilon}_{i-1,i+l-1}}_l(m_i,\hdots,m_{i+l-1}),
m_{i+l},\hdots,m_n) =0,
\end{equation*}
where $\dagger=|m_1|+\cdots+|m_{i-1}|+(i-1)$, and $m_1,\hdots,m_n
\in M$ are of homogeneous degree. In other words, the above
operations form the morphisms and higher operations of an
$A_\infty$-category over $A$ whose objects consist of the
augmentations $\varepsilon \colon (\mathcal{A},\partial) \to (A,0)$.
\end{Thm}
\begin{proof}
The proof follows mutatis mutandis from the proof of Theorem
\ref{thm:case1}. To that end, we just have to check the fact that
the equality
\[ \id^{\boxtimes i-1} \boxtimes \mu_l \boxtimes \id^{\boxtimes (k-i)} = (\id^{\boxtimes i-1} \boxtimes \partial_n \boxtimes \id^{\boxtimes (k-i)})^*\]
is satisfied.
\end{proof}

\section{A toy example of a DGA}\label{sec: Lenny's example}
In this section we discuss a toy example which illustrates the two
different $A_\infty$-structures defined. This example was suggested
by Lenny Ng and was an inspiration for this paper.

Let $A$ be an algebra over $R= \Z/2\Z$ and let $g_1, g_2$ be two
elements of $A$ which do not (necessarily) commute. We will consider
the differential graded algebra $({\mathcal A}, \partial)$ over $A$
generated by $c_1, \ldots,c_5$ and with differential
\begin{eqnarray*}
& & \partial c_1 = c_2 g_1 c_4 +c_3,\\
& & \partial c_2 = c_5 g_2,\\
& & \partial c_3 = c_5 g_2 g_1 c_4,\\
& & \partial c_4 = \partial c_5 =0.
\end{eqnarray*}
It is easily checked that $\partial^2=0$. Moreover, $\partial_0$
vanishes and therefore there is a canonical augmentation which sends
every generator to zero.

\subsection{The $A_\infty$-structure defined in Subsection
\ref{subsec:case1} (Case I)} The construction of Subsection
\ref{subsec:case1}, performed on the trivial augmentation, gives
rise to an $A_\infty$-algebra structure on $M^\vee=\oplus_{i=1}^5 A
c_i$. Let $a, a' \in A$. Using Lemma \ref{explicit computation I} we
compute the first order operations
\begin{eqnarray*}
& & \mu_1(ac_1)=\mu_1(ac_2)=\mu_1(ac_4)=0,\\
& & \mu_1(ac_3)=ac_1,\\
& & \mu_1(ac_5)=ag_2\cdot c_2,
\end{eqnarray*}
and the second order operations
\begin{eqnarray*}
& & \mu_2(ac_2,a'c_4)=ag_1a' \cdot c_1, \\
& & \mu_2(ac_5,a'c_4)=ag_2g_1a' \cdot c_3,
\end{eqnarray*}
while $\mu_2(ac_i,a'c_j)=0$ whenever $(i,j) \notin \{(2,4),(5,4)\}$.
Finally, $\mu_n \equiv 0$ for all $n \ge 3$.

We verify that these operations verify the $A_\infty$ relations. The
only nontrivial relation to verify (i.e. the only one where not all
terms vanish) is
\begin{eqnarray*}
\lefteqn{\mu_1(\mu_2(a_1c_5, a_2c_4))+ \mu_2(\mu_1(a_1c_5), a_2c_4)+
\mu_2(a_1c_5,
\mu_1(a_2c_4)) =} \\
&=& \mu_1(a_1g_2g_1a_2c_3) + \mu_2(a_1g_2c_2, a_2c_4) + 0 \\
&=& a_1g_2g_1a_2c_1 + a_1g_2g_1a_2c_1\\
&=& 0.
\end{eqnarray*}
\subsection{The $A_\infty$-structure defined in Subsection
\ref{subsec:case2} (Case II)} In this subsection we assume that $A$
is a Hermitian algebra. The construction of Subsection
\ref{subsec:case2}, performed on the trivial augmentation, gives
rise to an $A_\infty$-algebra structure on $M$, which we identify to
a submodule of $M^*$ by \eqref{inclusion in dual}. Elements in
$M^{\boxtimes n}$ can be written, as usual, as linear combinations
of terms of the form $a_0c_{i_1}a_1 \ldots a_{n-1}c_{i_n}a_n$, where
$a_0, \ldots a_n \in A$ and $c_{i_1}, \ldots, c_{i_n}$ are element
of the prescribed basis of $M$. Since in case II the operations
$\mu_n$  are morphisms of $A$--$A$-bimodule, we will give their
values only on elements of the form $c_{i_1}a_1 \ldots
a_{n-1}c_{i_n}$.

Using Lemma \ref{explicit computation II} we compute the first order
operations
\begin{eqnarray*}
& & \mu_1(c_1)=\mu_1(c_2)=\mu_1(c_4)=0,\\
& & \mu_1(c_3)=c_1,\\
& & \mu_1(c_5)=c_2 g_2^\star
\end{eqnarray*}
and the second order operations, for all $h \in A$,

\begin{eqnarray*}
& & \mu_2(c_2hc_4)= \mathfrak{t}(h, g_1)c_1, \\
& & \mu_2(c_5hc_4)= \mathfrak{t}(h, g_2g_1) c_3,
\end{eqnarray*}
while $\mu_2(c_ihc_j)=0$ in all other cases. Finally, $\mu_n \equiv
0$ for all $n \ge 3$.

The only nontrivial $A_\infty$ relation to check is
\begin{eqnarray*}
\lefteqn{\mu_1(\mu_2(c_5hc_4))+ \mu_2(\mu_1(c_5) \boxtimes hc_4) +
\mu_2(c_5h \boxtimes
\mu_1(c_4)) = } \\
& = & \mathfrak{t}(h, g_2g_1)\mu_1(c_3) + \mu_2(c_2g_2^\star hc_4) +  0 \\
& = & \mathfrak{t}(h, g_2g_1) c_1 + \mathfrak{t}(g_2^\star h, g_1)c_1 \\
& = & 0
\end{eqnarray*}
because $\mathfrak{t}(h, g_2g_1) = \mathfrak{t}(g_2^\star h, g_1)$
by the properties of the adjoint.

\section{Potential examples of knots distinguished by the constructed $A_\infty$-structures}
\label{sec:computation}
For computational purposes, the $A_\infty$-algebra  is much easier to use for extracting invariants compared to the DGA. For instance, the products and higher order Massey products in linearised Legendrian cohomology introduced in \cite{Productstructure} were in the same article shown to be efficient tools for distinguishing a Legendrian knot from its mirror (in case when the underlying homologies are isomorphic). The latter construction considered an $A_\infty$-structure for coefficients in $\Z_2$.

Assume that there exists a Legendrian knot $\Lambda_n \subset (\R^3,dz-ydx)$ which satisfies the following for some $n>1$:
\begin{enumerate}[label=(\roman*)]
\item The rotation number of $\Lambda_n$ is zero;
\item The bound on the Thurston-Bennequin invariant of $\Lambda_n$ in terms of the Kauffman polynomial of the underlying smooth knot is not sharp. In particular, this means that the Chekanov-Eliashberg algebra of $\Lambda_n$ has no augmentation in $\Z_2$ (see \cite{RutherfordKauffman} for more details); and
\item The Chekanov-Eliashberg algebra of $\Lambda_n$ admits a (0-graded) augmentation in $M_n(\Z_2)$.
\end{enumerate}
The authors expect that such knots can be constructed using the methods from \cite[Theorem 4.8]{Satellites} but, unfortunately, as of now they are not aware of an explicit example. In any case, performing cusp connected sums (see \cite{Etnyre_&_Connected_Sums}) between the Legendrian knots considered in \cite{Productstructure} and the hypothetical Legendrian knot $\Lambda_n$ we obtain Legendrian knots for which our construction can be used as an efficient computational tool; see Proposition \ref{prp:computation} below.

First we recall some details concerning the examples from
\cite{Productstructure}. For a Legendrian knot $\Lambda$ we denote
by $\overline{\Lambda}$ its \emph{Legendrian mirror}, i.e.~its image
under the contactomorphism $(x,y,z) \mapsto (x,-y,-z)$. Consider the
examples $\Lambda_{k,l,m}$ constructed in the proof of \cite[Theorem
1.1]{Productstructure} (the first part) which satisfy the following.
For any triple $k,l,m \ge 1$ it is the case that
\begin{itemize}
\item the rotation number of $\Lambda_{k,l,m}$ is zero, and the DGA is hence graded in the integers, and
\item given that the three numbers $l-m-1$, $m-k+1$, $l-k+1$, are distinct, there is a unique graded augmentation $\varepsilon$ in $\Z_2$.
\end{itemize}
It follows that the same properties are satisfied for its Legendrian mirror $\overline{\Lambda_{k,l,m}}$. Given $k,l,m$ satisfying the second property, the knot $\Lambda_{k,l,m}$ is distinguished from its mirror up to Legendrian isotopy by a computation showing that
\begin{enumerate}
\item the product
$$ \mu^{\varepsilon,\varepsilon,\varepsilon}_2 \colon LCH^{l-m-1}_\varepsilon(\Lambda_{k,l,m}) \oplus LCH^{m-k+1}_\varepsilon(\Lambda_{k,l,m}) \to LCH^{l-k}_\varepsilon(\Lambda_{k,l,m})$$
does not vanish identically, while
\item the product
$$ \mu^{\varepsilon,\varepsilon,\varepsilon}_2 \colon LCH^{l-m-1}_\varepsilon(\overline{\Lambda_{k,l,m}}) \oplus LCH^{m-k+1}_\varepsilon(\overline{\Lambda_{k,l,m}}) \to LCH^{l-k}_\varepsilon(\overline{\Lambda_{k,l,m}})$$
does vanish.
\end{enumerate}
We can use these examples to show the following.
\begin{Prop}
\label{prp:computation}
Assume the existence of a Legendrian knot $\Lambda_n$ for some $n>1$ satisfying conditions (i)--(iii) above. The cusp connected sum $\Lambda_{k,l,m} \# \Lambda_n \subset (\R^3,dz-ydx)$ is a Legendrian knot admitting a (0-graded) augmentation in $M_n(\Z_2)$ but not in $\Z_2$. For suitable $k,l,m > 0$ (depending on the knot $\Lambda_n$) the Legendrian knot $\Lambda_{k,l,m} \# \Lambda_n$ can moreover be distinguished from both $\overline{\Lambda_{k,l,m}} \# \Lambda_n$ and $\overline{\Lambda_{k,l,m} \# \Lambda_n}$ using the $A_\infty$-structure in linearised Legendrian contact homology with coefficients in $M_n(\Z_2)$ as constructed in Section \ref{subsec:case2} (i.e.~Case II).
\end{Prop}

\begin{proof}[Sketch of proof]
The fact that the connected sum has augmentations in $M_n(\Z_2)$ but not in $\Z_2$ was shown in \cite[Lemma 4.3]{EstimNumbrReebChordLinReprCharAlg}. The existence part uses the following explicit construction that we now outline. Given augmentations $\varepsilon_i \colon \Lambda_i \to A_i$, $i=1,2$, in the unital algebras $A_i$, then there is an induced augmentation $(\varepsilon_1 \# \varepsilon_2) \colon \mathcal{A}(\Lambda_1 \# \Lambda_2) \to A_1 \otimes A_2$ determined as follows. Recall that
$$\mathcal{A}(\Lambda_1 \# \Lambda_2)=\mathcal{A}(\Lambda_1) \star \mathcal{A}(\Lambda_2) \star \langle c_0 \rangle$$
holds on the level of generators, where $|c_0|=0$. The induced augmentation is determined uniquely by the requirements that $(\varepsilon_1 \# \varepsilon_2)(c)=\varepsilon_i(c)$ holds on the old generators (using the canonical algebra maps $a \mapsto a \otimes 1_{A_2} \in A_1 \otimes A_2$ and $b \mapsto 1_{A_1} \otimes b \in A_1 \otimes A_2$) while $(\varepsilon_1 \# \varepsilon_2)(c_0)=1=1_{A_1} \otimes 1_{A_2}  \in A_1 \otimes A_2$ holds on the new generator.

The computations of the DGA of $\Lambda_{k,l,m}$ performed in \cite{Productstructure} can readily be seen to give the following. Consider the construction of the $A_\infty$-structure with coefficients in $M_n(\Z_2)$ as defined in Section \ref{subsec:case2} (i.e.~Case II). The non-vanishing of the product as in (1) again holds for $\Lambda_{k,l,m} \# \Lambda_n$ when using the augmentation $\varepsilon \# \varepsilon_2$ taking values in $M_n(\Z_2)$. It can moreover be seen that that (2) is satisfied for \emph{any} pair of graded augmentations in $M_n(\Z_2)$ for the same coefficients, given that $k,l,m >0$ were chosen appropriately. E.g.~we can choose $k,l,m>0$ so that $l-m-1$, $m-k+1$, $l-k+1$ all are distinct and sufficiently large (depending on the degrees of the Reeb chords of $\Lambda_n$).
\end{proof}
The following result shows the relation between the linearised Legendrian contact cohomology of a Legendrian knot and its Legendrian mirror.
\begin{Lem}
Let $\Lambda \subset (\R^3,dz-ydx)$ be a Legendrian knot. For any pair of augmentations $\varepsilon_i \colon (\mathcal{A}(\Lambda),\partial) \to M_n(R)$, $i=1,2$, there are induced augmentations $\overline{\varepsilon}_i \colon (\mathcal{A}(\overline{\Lambda}),\partial') \to M_n(R)$ for which there is a canonical isomorphism
$$ (LCC^\bullet(\Lambda),d^{\varepsilon_0,\varepsilon_1}) \simeq (LCC^\bullet(\overline{\Lambda}),d^{\overline{\varepsilon}_1,\overline{\varepsilon}_0})$$
of graded $R$-bimodules. (This can even be made into an isomorphism of free $M_n(R)$-bimodules, but in this case we must use a non-standard free bimodule structure on the latter where left and right multiplication has been interchanged, while utilising the transpose of a matrix.)
\end{Lem}
\begin{proof}
Recall that there is a canonical grading-preserving bijection between the set of generators of $\Lambda$ and $\overline{\Lambda}$. Under the corresponding identification $\mathcal{A}(\Lambda) \simeq \mathcal{A}(\overline{\Lambda})$ the differential of the latter takes the form $\partial'(c)=\iota \circ \partial(c)$ on the generators, where $\iota$ is the involution which reverses the letters in each word (this is an isomorphism from a free algebra to its opposite). The statement readily follows if we take $\overline{\varepsilon}_i$ to be defined by
\[ \overline{\varepsilon}_i(c):=\left(\varepsilon_i(c)\right)^t,\]
the latter denoting the transpose of a matrix in $M_n(R)$ (this is also an involution inducing an isomorphism from the ring of matrices to its opposite ring).
\end{proof}

\section{Directed systems and consistent sequences of DGAs}
Both directed systems and consistent sequences of differential
graded algebras appear naturally in applications. In this section we
discuss briefly how our constructions can be carried over to these
cases.

\subsection{The infinitely generated case: a directed system
of DGAs} The differential graded algebra considered up to this point
have all been finitely generated. For many of the applications that
we have in mind this is also sufficient. Namely,  the
Chekanov-Eliashberg DGA of a Legendrian submanifold $\Lambda$ is
generated by the Reeb chords of $\Lambda$, and a generic Legendrian
submanifold has finitely many Reeb chords in most contact manifolds
for which the Chekanov-Eliashberg DGA is rigorously defined. For
example, this is the case for closed Legendrian submanifolds of the
standard contact $\R^{2n+1}$.

Nonetheless, for a general contact manifold there may be infinitely
many Reeb chords on a generic closed Legendrian submanifold. In this
case then the Chekanov-Eliashberg DGA is infinitely generated, and
hence $M$ is a free $A$--$A$-bimodule with an infinite preferred
basis. However, note that to every Reeb chord we can associate an
\emph{action} $\ell \in \R_{>0}$, and generically all Reeb chords
below a certain action still comprise a finite subset. We write
$M^{\ell} \subset M$ for the free and finitely generated
$A$--$A$-bimodule spanned by the Reeb chords of action less than
$\ell>0$. We write $\mathcal{A}^\ell:= \mathcal{T}_A(M^{\ell})$, and
the action-decreasing property of the differential in the
Chekanov-Eliashberg DGA implies that each ${\mathcal A}^\ell$ is a
sub-DGA of ${\mathcal A}$, and therefore there is an induced
directed system
\[i_{\ell_1,\ell_2} \colon (\mathcal{A}^{\ell_1},\partial) \hookrightarrow
(\mathcal{A}^{\ell_2},\partial), \quad \ell_1 \le \ell_2,\] of
finitely generated differential graded algebras. The direct limit of
this directed system is the infinitely generated differential graded
algebra $(\mathcal{A}, \partial)$ and therefore we can reduce the
study of an infinitely generated graded algebra endowed with an
``action filtration'' as above to the study of directed systems of
finitely generated DGAs.

In this setting the $A_\infty$-categories obtained by applying the
constructions in Subsection \ref{subsec:case1} and
\ref{subsec:case2} to the direct system $({\mathcal A}^\ell,
\partial)$ form an inverse system; namely we have morphisms
\begin{gather*}
i_{\ell_1,\ell_2}^\vee \colon (M^{\ell_2})^\vee \to (M^{\ell_1})^\vee, \quad \ell_1 < \ell_2, \\
i_{\ell_1,\ell_2}^* \colon (M^{\ell_2})^* \to (M^{\ell_1})^*, \quad
\ell_1 < \ell_2.
\end{gather*}
Using the given choice of basis of $M$, the adjoint morphisms
$i_{\ell_1,\ell_2}^\vee$ and $i_{\ell_1,\ell_2}^*$ both correspond
to canonical projections onto the submodules spanned by the
generators having actions at most $\ell_1$. The linearised
coboundary maps $\mu^{\varepsilon_0,\varepsilon_1}$ (defined using
either of the constructions) makes the above inverse systems into
inverse systems of complexes; i.e.~the above projection maps are
chain maps. The Mittag-Leffler property be seen to hold for the
corresponding inverse system of boundaries, and hence the inverse
limits of homologies is equal to the homology of the inverse limit
complex.

The respective $A_\infty$-structures constructed for the above
inverse system of complexes can then seen to satisfy
$i_{\ell_1,\ell_2}^\vee \circ \mu_n=\mu_n \circ
((i_{\ell_1,\ell_2}^\vee)^{\otimes n})$ and $i_{\ell_1,\ell_2}^*
\circ \mu_n=\mu_n \circ ((i_{\ell_1,\ell_2}^*)^{\boxtimes n})$. This
gives rise to an $A_\infty$-structure on the inverse limits of
$(M^{\ell})^\vee$ and $(M^{\ell})^*$.

\subsection{Consistent sequences of DGAs}
\label{sec:cons-sequ-dga-1}

The construction of $\mathfrak{A}_n$ and $\mathcal{A}_{A_n}$ in
Section \ref{sec:semisimple} out of $\mathcal{A}$ leads to families
of differential graded algebras with an increasing number of
generators. Such sequences were used in \cite{augcat} to upgrade the
$A_\infty$ algebra structure from \cite{Productstructure} to an
$A_\infty$ category whose objects are the augmentations of
$\mathcal{A}$. This idea was later generalised in \cite{NRSSZ},
where the notion of a \textit{consistent} family of differentiable
graded algebras was introduced. Here we briefly describe this notion
and show how it also gives rise to $A_\infty$-categories with
noncommutative coefficients. The geometrical construction underlying
this algebraic definition will be sketched in Appendix
\ref{sec:cons-sequ-dga}.

Let $(\mathcal{A},\partial)$ be a semifree differential graded
algebra over the noncommutative algebra $A$. Its underlying algebra
is thus the tensor algebra $\mathcal{T}_A(M)$ over $A$ of a free
$A$--$A$-bimodule $M$ with basis ${\mathcal B}$. An
\textit{(m-components) link  grading} (as introduced in
\cite{Mis_grading}) on $\mathcal{A}$ is a pair of maps $b,e \colon
{\mathcal B} \rightarrow \{1,\hdots, m\}$ such that:
\begin{itemize}
\item If $c \in {\mathcal B}$ is such that $b(c)\not = e(c)$ then
$\partial(c)$ has no constant term, and
\item For any $c \in {\mathcal B}$ and any word $a_0c_1a_1\cdots
  a_{n-1}c_na_n$ appearing in an expression of $\partial(c)$, we have
  $b(c_{i-1})=e(c_{i})$.
\end{itemize}

A generator $c \in {\mathcal B}$ is called \textit{pure} if
$b(c)=e(c)$ and {\em mixed} otherwise. On $\mathfrak{A}_n$ and
${\mathcal A}_{A_n}$ there is a link grading defined by
$b(c^{ij})=i$ and $e(c^{ij})=j$. Moreover, words $a_0c_1a_1\cdots
a_{n-1}c_na_n$ such that $b(c_{i-1})=e(c_{i})$ (i.e. of the type
appearing in the differential of a basis element) are called
\textit{composable} in \cite{EffectLegendrian}, \cite{augcat} and
\cite{NRSSZ}. This terminology comes from the Chekanov-Eliashberg
algebra of an $m$-components Legendrian link: the components are
labeled by $\{1,\hdots,m\}$, and the maps $b$ and $e$ give the label
of the component of the starting point and endpoint of a Reeb chord
of the link. Composable words are those which can appear as negative
asymptotics of a holomorphic disc with boundary on the cylinder over
the link.

Let $(\mathcal{A},\partial)$ be a differential graded algebra
equipped with a link grading $(b,e)$ and let $I$ be a subset of
$\{1,\hdots,m\}$. We denote by $\mathcal{A}_I$ the subalgebra of
$\mathcal{A}$ generated by basis elements $c$ for which
$b(c),e(c)\in I$. There is a projection $\pi \colon {\mathcal A} \to
{\mathcal A}_I$ such that, for every basis element, $\pi(c) = c$ if
$(b(c), e(c)) \in I \times I$, and $\pi(c)=0$ otherwise. It follows
from the definition of a link grading that $\partial$ descends to a
differential $\partial_I = \pi \circ \partial$ on $\mathcal{A}_I$.
For $m$-components Legendrian links this corresponds to taking
chords of the sub-link whose components are labeled by $I$ and
defining a differential which counts only holomorphic discs which
are asymptotic to
 chords in this sublink.
Note that the differential graded algebra $\mathcal{A}_I$ is
equipped with a link grading once we identify $I$ with $\{1,\hdots,
l\}$ by an order preserving identification. When $I=\{i\}$ we denote
$\mathcal{A}_I$ simply by $\mathcal{A}_i$.

We give now the definition of a consistent family of differential
graded algebras following \cite{NRSSZ}.
\begin{defn}
A sequence $(\mathcal{A}^{(i)},\partial^{(i)})$ of semi-free
differential graded algebras with generating sets ${\mathcal
B}^{(i)}$ and link gradings $(b^{(i)},e^{(i)})$ taking values in
$\{1, \ldots, i \}$ is \emph{consistent} if the following properties
are satisfied.
\begin{enumerate}
\item For every increasing map $f \colon \{1,\hdots, i\}\rightarrow \{1,\hdots,j\}$
there is an induced map $h_f \colon {\mathcal B}^{(i)}\rightarrow
{\mathcal B}^{(j)}$ such that, for any generator $c \in {\mathcal
B}^{(i)}$, we have
$$(b^{(j)}(h_f(c)),e^{(j)}(h_f(c)))=(f(b^{(i)}(c)),f(e^{(i)}(c))).$$
\item For any two composable increasing maps $f$ and $g$ between finite sets,
we have $h_{f\circ g}=h_f\circ h_g$.
\item For any $f$ as above, the algebra morphism $h_f \colon \mathcal{A}^{(i)}
\rightarrow \mathcal{A}^{(j)}$ coinciding with $h_f$ on generators
satisfies the property that $\pi\circ
h_f:\mathcal{A}^{(i)}\rightarrow \mathcal{A}^{(j)}_{f(\{1,\hdots
i\})}$ is an isomorphism of differential graded algebras.
\end{enumerate}
\end{defn}
Note that increasing maps from $\{1,\hdots, i\}$ to $\{1,\hdots,
j\}$ are in one-to-one correspondence with subsets of $\{1,\hdots,
j\}$ of cardinality $i$.  Figure \ref{fig:3-6copy}
shows the geometric meaning of the maps $h_f$ when ${\mathcal
A}^{(i)}$ is the Chekanov-Eliashberg of the $i$-copy link of a
Legendrian submanifold. See Appendix \ref{sec:posit-augm-repr}.

The upshot of this definition is that, since $\mathcal{A}^{(1)}$ is
isomorphic to $\mathcal{A}^{(i+1)}_{k}$ for any $k\in
\{1,\hdots,i+1\}$, any $(i+1)$-tuple of augmentations
$(\varepsilon_0, \ldots, \varepsilon_i)$ of $\mathcal{A}^{(1)}$
gives rise to an augmentation $\boldsymbol{\varepsilon}$ of
$\mathcal{A}^{(i+1)}$ which vanishes on the mixed generators and
satisfies $\boldsymbol{\varepsilon}(c) = \varepsilon_k(c)$ for any
$c \in {\mathcal A}^{(i+1)}_k \cong {\mathcal A}^{(1)}$.

\begin{figure}[t]
\labellist \pinlabel $d_{2,1}$ at 140 278 \pinlabel $h_{\{2,5,6\}}$
at 370 294 \pinlabel $d_{5,2}$ at 579 278
\endlabellist
  \centering
  \includegraphics[width=\textwidth]{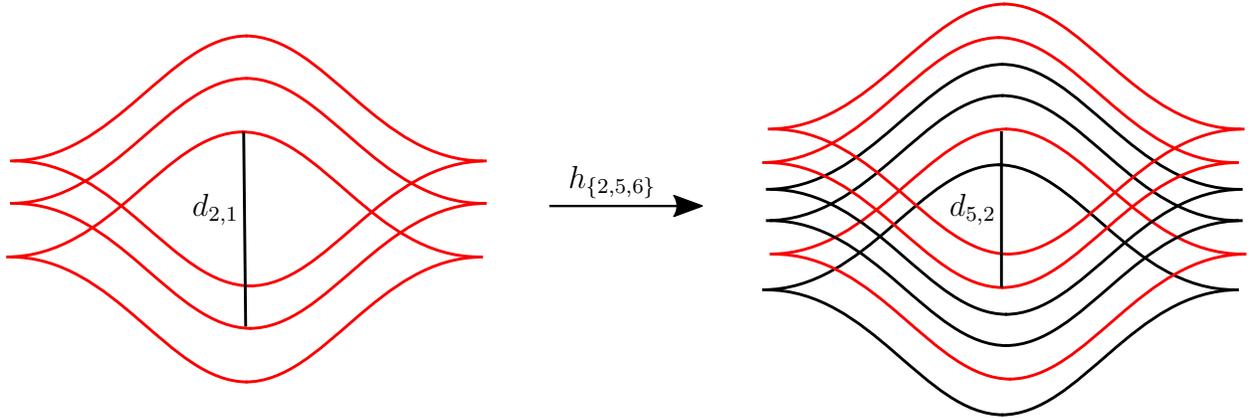}
  \caption{The $3$-copies link as a sublink of the $6$-copy link.}
  \label{fig:3-6copy}
\end{figure}

We denote by $M^{(i+1)}$ the free bimodule generated by ${\mathcal
B}^{(i+1)}$. Also for any subset $I$ of $\{1,\hdots,j\}$ of
cardinality $i+1$, we denote by $M^{I}$ the corresponding submodule
of $M^{(j)}$ (which is identified with $M^{(i+1)}$). We decompose
each differential $\partial^{(i+1)}$ restricted to $M^{(i+1)}$ into
a sum
$$\partial^{(i+1)}|_{M^{(i+1)}}= \partial_0^{(i+1)}+\cdots +\partial_k^{(i+1)},$$
where $\partial_l^{(i+1)}$ takes values in $(M^{(i+1)})^{\boxtimes
l}$. Now, given an $(i+1)$-tuple of augmentations of
$\mathcal{A}^{(1)}$, inducing an augmentation
$\boldsymbol{\varepsilon}$ of $\mathcal{A}^{(i+1)}$, we define the
operation $\mu^{\epsilon_0, \ldots, \epsilon_i}_{i}$ as follows. We
consider the map

\begin{equation}
\label{eq:3} \xymatrixrowsep{0.2in} \xymatrixcolsep{0.5in}
\xymatrix{
      M^{(2)}\ar[r]^{h_{\{1,i+1\}}} & M^{(i+1)} \ar[r]^{(\partial^{(i+1)})^{\boldsymbol{\varepsilon}}_i} & (M^{(i+1)})^{\boxtimes i} \ar[d]^{\pi} \\
& & M^{(i+1)}_{\{i,i+1\}}\boxtimes M^{(i+1)}_{\{i-1,i\}}\boxtimes
\ldots
\boxtimes M^{(i+1)}_{\{1,2\}} \ar[d]^{\simeq} \\
& & (M^{(2)})^{\boxtimes i}.}
\end{equation}

The map $\pi \colon (M^{(i+1)})^{\boxtimes i} \to
M^{(i+1)}_{\{i,i+1\}} \boxtimes \ldots \boxtimes
M^{(i+1)}_{\{1,2\}}$ is the restriction of the canonical projection
$\pi \colon \mathcal{A}^{(i+1)} \to
\mathcal{A}^{(i+1)}_{\{i,i+1\}}\star\ldots\star\mathcal{A}^{(i+1)}_{\{1,2\}}$.

This allows us to define $A_\infty$-categories whose objects are
augmentations of $\mathcal{A}^{(1)}$, the morphism space between any
pair of augmentation is a copy of $M^{(2)}$, and the compositions
are defined by taking adjoints of the maps $M^{(2)} \to
(M^{(2)})^{\boxtimes i}$ defined in \eqref{eq:3}, using the
construction from either Subsection \ref{subsec:case1} or
\ref{subsec:case2}.

Note that the procedure described in Section \ref{sec:semisimple}
which associates to $\mathcal{A}$ and $n$ the differential graded
algebra $\mathfrak{A}_n$ produces a consistent sequence of
differential graded algebras whose link grading is
$(b(c^{ij}),e(c^{ij}))=(i,j)$. This sequence has the property that
$M^{(2)} \cong M^{(1)} = M$, and therefore the augmentation category
defined from it contains the same information as the differential
graded algebra ${\mathcal A}$. However, there exist consistent
sequences containing strictly more information than simply that
contained in $\mathcal{A}=\mathcal{A}^{(1)}$. For instance, even in
the case when ${\mathcal A}$ is finitely generated, an infinite
consistent sequence may still give rise an $A_\infty$-category with
nontrivial operations of arbitrarily high order. In Appendix
\ref{sec:cons-sequ-dga} we sketch the geometric construction of
\cite{NRSSZ}, which illustrates such a phenomenon.

\appendix
\section{The geometric setting}

In this appendix, we  discuss the geometric motivation of our
constructions.

\subsection{The Legendrian contact homology with twisted
coefficients} \label{sec:legendr-cont-homol}

Here we give a very brief sketch of the construction of the
differential graded algebras that arise in Legendrian contact
homology. We refer to \cite{Chekanov_DGA_Legendrian} and
\cite{LCHgeneral} for more details. Let  $\Lambda \subset
(Y,\alpha)$ be a Legendrian submanifold in a manifold $Y$ with a
contact one-form $\alpha$. The contact one-form induces the Reeb
vector field on $Y$. A \emph{Reeb chord} of $\Lambda$ is an integral
curves of the Reeb vector field with  both endpoints on $\Lambda$.
Generically, the Reeb chords form a discrete set.

We denote $M_R(\Lambda)$ the free graded $R$-module generated by the
Reeb chords of $\Lambda$. The grading is induced by the
Conley-Zehnder index of the chords and, in general, takes values in
a cyclic group. For now on we assume that there is only a finite
number of Reeb chords.

In the most basic setting, Legendrian contact homology associates a
differential graded algebra structure on the tensor algebra
$\mathcal{T}_R(M_R(\Lambda))$ over $R$; this is the
Chekanov-Eliashberg DGA. The differential of a Reeb chord $d_0$
counts the rigid pseudoholomorphic punctured discs in the
symplectisation $\R \times Y$, having boundary on the Lagrangian
cylinder $\R \times \Lambda$ over $\Lambda$ and one positive
strip-like end asymptotic to $d_0$. For example, a pseudoholomorphic
disc as shown in Figure \ref{fig:curve} gives a contribution to the
coefficient in front of the word $d_1d_2d_3d_4d_5 \in
(M_R(\Lambda))^{\otimes 5}$ in the expression of $\partial(d_0)$.

In \cite{Albin} Eriksson-\"{O}stman extends the definition of the
Chekanov-Eliashberg DGA to a version with coefficients in the group
ring $A:=R[\pi_1(\Lambda)]$. In this case the underlying graded
algebra is the tensor algebra $\mathcal{T}_A(M_A(\Lambda))$, where
$M_A(\Lambda)$ is the free graded $A$--$A$-bimodule generated by the
Reeb chords of $\Lambda$. Fix a choice of capping paths for each
end-point of a Reeb chord, i.e.~a path in $\Lambda$ connecting the
end-point with a given base point. Roughly speaking, the
differential then takes the following form. Assume that
$a_0,a_1,\hdots,a_5 \in \pi_1(\Lambda)$ in Figure \ref{fig:curve}
denote the homotopy classes of closed curves corresponding to the
canonically oriented boundary arcs in $\R \times \Lambda$ of a
pseudoholomorphic disc contributing to the differential, where each
boundary arc has been closed up by using the corresponding capping
path. Then the depicted disc contributes to the $R$-coefficient in
front of $a_0 d_1 a_1 d_2a_2d_3a_3d_4a_4d_5a_5 \in
(M_A(\Lambda))^{\boxtimes 5}$ in the expression $\partial(d_0)$.

\begin{Rem}
Note that the above construction requires that we introduce
auxiliary capping paths in $\Lambda$ connecting each end-point of a
Reeb chord with a given base point. It would be more natural to
replace the fundamental group by the fundamental groupoid, and
consider differential graded algebras over a groupoid algebra. This
generalisation does not require any new idea, but for simplicity of
notation we will only consider group algebras.
\end{Rem}

\subsection{$\mathcal{A}_\infty$-categories of a Legendrian link.}
\label{sec:cons-sequ-dga} The construction of Section
\ref{sec:cons-sequ-dga-1} allows us to define several versions of
$\mathcal{A}_\infty$-categories coming from various consistent
sequences of DGAs.

\subsubsection{The ``negative'' augmentation and representation
categories.} \label{sec:negat-augm-repr}

Let $\Lambda$ be a Legendrian submanifold such that its
Chekanov-Eliashberg DGA (over the appropriate coefficient algebra)
admits augmentations. Then, by applying the constructions of
Subsections \ref{subsec:case1} and \ref{subsec:case2}, we associate
to $\Lambda$ three $A_\infty$-categories which generalise the
augmentation category of \cite{augcat} to the noncommutative
setting. In the terminology of \cite{NRSSZ} these are the
``negative'' augmentation categories.

\paragraph{\textbf{The category $\mathcal{A}\mathrm{ug}_-(\Lambda,R[\pi_1(\Lambda)])$.}}
\label{sec:aug_-}

Let $\mathcal{A}(\Lambda)$ be the Chekanov-Eliashberg algebra of
$\Lambda$ as defined in Appendix~\ref{sec:cons-sequ-dga} with
coefficients in $R[\pi_1(\Lambda)]$. We define
$\mathcal{A}\mathrm{ug}_-(\Lambda,R[\pi_1(\Lambda)])$ as the
$\mathcal{A}_\infty$-category such that:
\begin{enumerate}
\item $\mathcal{O}b(\mathcal{A}\mathrm{ug}_-(\Lambda,R[\pi_1(\Lambda)]))$
is the set of augmentations  $\varepsilon \colon \mathcal{A} \to
R[\pi_1(\Lambda)]$.
\item For every pair of augmentations $\varepsilon_1,\varepsilon_2\in
  \mathcal{O}b(\mathcal{A}\mathrm{ug}_-(\Lambda,R[\pi_1(\Lambda)]))$, the morphism space
  $\Hom(\varepsilon_1,\varepsilon_2)$ is the free
  $R[\pi_1(\Lambda)]$--$R[\pi_1(\Lambda)]$-bimodule generated by the Reeb
  chords of $\Lambda$.
\item The operations $\mu_n$, $n \ge 1$, are the
  $R[\pi_1(\Lambda)]$--$R[\pi_1(\Lambda)]$-bimodule
  maps $$\mu_n:\Hom(\varepsilon_{n-1},\varepsilon_{n})\boxtimes
  \Hom(\varepsilon_{n-2},\varepsilon_{n-1})\boxtimes\cdots\boxtimes
  \Hom(\varepsilon_0,\varepsilon_1)\rightarrow
  \Hom(\varepsilon_0,\varepsilon_n)$$ defined in Subsection
  \ref{subsec:case2}.
\end{enumerate}

The fact that $\mathcal{A}\mathrm{ug}_-(\Lambda,R[\pi_1(\Lambda)])$
is an $\mathcal{A}_\infty$-category comes from Theorem
\ref{thm:caseII}. Let $d_2\in\Hom(\varepsilon_2,\varepsilon_3)$ and
$d_4\in\Hom(\varepsilon_1,\varepsilon_2)$. Assuming that
$\varepsilon_2(d_3)=\sum a_kh_k$, the disc in Figure \ref{fig:aug-}
gives a contribution of
$a_kg_1^{-1}\varepsilon_3(d_1)^*g_0^{-1}d_0g_5^{-1}\varepsilon_1(d_5)^*g_5^{-1}$
to $\mu_2(d_2 h_kd_4)$ for all $k$.

\begin{figure}[ht]
\labellist \pinlabel $d_0$ at 100 123 \pinlabel $\varepsilon_3(d_1)$
at 10 -6 \pinlabel $d_2$ at 55 -6 \pinlabel $\varepsilon_2(d_3)$ at
101 -6 \pinlabel $d_4$ at 145 -6 \pinlabel $\varepsilon_1(d_5)$ at
190 -6 \pinlabel $g_0$ at 33 65 \pinlabel $g_1$ at 32 14 \pinlabel
$g_2$ at 78 14 \pinlabel $g_3$ at 122 14 \pinlabel $g_4$ at 168 14
\pinlabel $g_5$ at 167 65
\endlabellist
  \centering
  \includegraphics{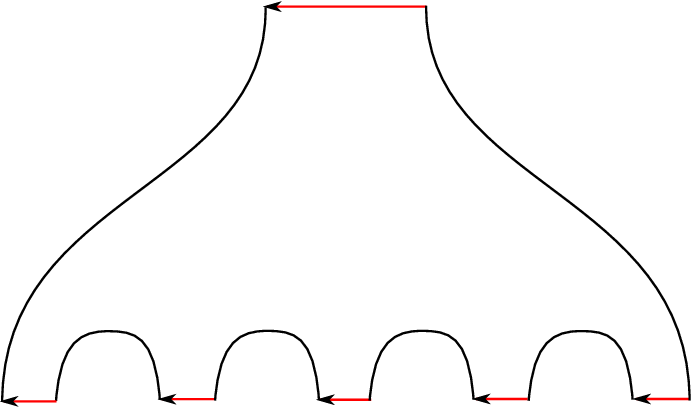}
  \vspace{5mm}
  \caption{A pseudoholomorphic disc contributing to $\mu_2 \colon \Hom(\varepsilon_2,
\varepsilon_3) \boxtimes\Hom(\varepsilon_1,\varepsilon_2)\rightarrow
\Hom(\varepsilon_1, \varepsilon_3)$.}
  \label{fig:aug-}
\end{figure}

\paragraph{\textbf{The category $\mathcal{R}\mathrm{ep}(\Lambda,m)$.}}
\label{sec:repre}
 Let $\mathcal{A}(\Lambda)$ be the Chekanov-Eliashberg algebra of $\Lambda$
with coefficients in $R$. For $m\in \mathbb{N}$, we denote by
$M_m(R)$ the algebra of $m \times m$ matrices with entries in $R$.
We define $\mathcal{R}\mathrm{ep}(\Lambda,m)$ as the
$\mathcal{A}_\infty$-category such that:
\begin{enumerate}
\item $\mathcal{O}b(\mathcal{R}\mathrm{ep}(\Lambda,m))$ is the set of
  augmentations $\rho \colon \mathcal{A} \to M_m(R)$ (called $m$-dimensional
  representations of $\Lambda$).
\item For every pair of augmentations $\rho_1,\rho_2\in
 \mathcal{O}b(\mathcal{R}\mathrm{ep}(\Lambda,m))$, the morphism space\\
$\Hom(\rho_1,\rho_2)$ is the free $M_m(R)$--$M_m(R)$-bimodule
generated by Reeb chords of  $\Lambda$.
\item The operations $\mu_n$, $n \ge 1$, are the $M_m(R)$--$M_m(R)$-bimodule
  maps $$\mu_n \colon \Hom(\rho_{n-1},\rho_{n})\boxtimes
  \Hom(\rho_{n-2},\rho_{n-1})\boxtimes\cdots\boxtimes
  \Hom(\rho_0,\rho_1)\rightarrow \Hom(\rho_0,\rho_n)$$ defined in
  Subsection \ref{subsec:case2}.
\end{enumerate}

 From Theorem~\ref{thm:caseII} it follows that $\mathcal{R}\mathrm{ep}(\Lambda,m)$ is an
$\mathcal{A}_\infty$-category. Assuming that $\rho_2(d_3)$ is the
matrix with coefficients $(a_{ij})$ and $E_{i,j}$  denotes the
elementary matrix with entries $(\delta_{i,j})$, then the
pseudoholomorphic disc in Figure \ref{fig:rep-} gives a contribution
of $a_{ij}\rho_3(d_1)^*d_0 \rho_1(d_5)^*$ to $\mu_2(d_2 E_{i,j}d_4)$
for all $i,j$.

\paragraph{\textbf{The category $\mathcal{R}\mathrm{ep}_-(\Lambda,S)$.}}
\label{sec:genrep} Given a noncommutative $R$-algebra $S$, we define
the category $\mathcal{R}\mathrm{ep}(\Lambda,S)$ such that:

\begin{figure}[ht]
\labellist \pinlabel $d_0$ at 100 123 \pinlabel $\rho_3(d_1)$ at 10
-6 \pinlabel $d_2$ at 55 -6 \pinlabel $\rho_2(d_3)$ at 101 -6
\pinlabel $d_4$ at 145 -6 \pinlabel $\rho_1(d_5)$ at 190 -6
\endlabellist
  \centering
  \includegraphics{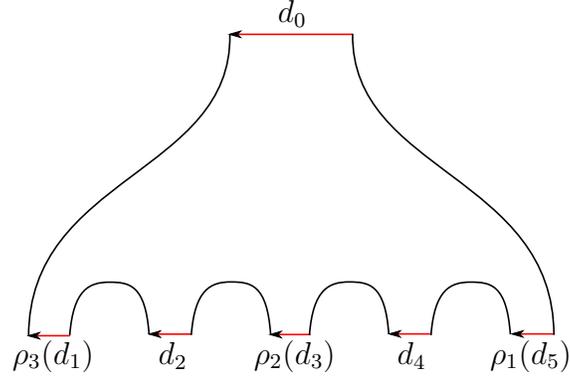}
  \vspace{5mm}
  \caption{A pseudoholomorphic disc contributing to $\mu_2 \colon \Hom(\rho_2,\rho_3)\boxtimes \Hom(\rho_1,\rho_2)\rightarrow \Hom(\rho_1,\rho_3)$.}
  \label{fig:rep-}
\end{figure}

\begin{enumerate}
\item $\mathcal{O}b(\mathcal{R}\mathrm{ep}(\Lambda,S))$ consists of
augmentations $\rho \colon \mathcal{A}(\Lambda) \to S$.
\item For every pair of augmentations $\rho_1,\rho_2\in
 \mathcal{O}b(\mathcal{R}\mathrm{ep}(\Lambda,m))$,  morphism space
$\Hom(\rho_1,\rho_2)$ is the free $S$-module
 generated by the Reeb chords of $\Lambda$
\item The operations $\mu_n$, $n \ge 1$, are the $R$-linear maps
$$\mu_n \colon \Hom(\rho_{n-1},\rho_{n})\otimes \Hom(\rho_{n-2},\rho_{n-1})\otimes\cdots
\otimes \Hom(\rho_0,\rho_1)\rightarrow \Hom(\rho_0,\rho_n)$$ defined
in Subsection \ref{subsec:case1}.
\end{enumerate}
We can also make hybrid constructions where the Chekanov-Eliashberg
algebra is defined over the group ring of $\pi_1(\Lambda)$ and the
augmentations take values in a matrix algebra $M_m(R)$ or in a more
general noncommutative algebra $S$.

\subsubsection{The ``positive'' augmentation and representation
categories.}\label{sec:posit-augm-repr}

The consistent sequence leading to
$\mathcal{A}\mathrm{ug}_-(\Lambda)$ is determined uniquely by
$\mathcal{A}(\Lambda)$ as described in \cite{augcat} (see the
discussion at the end of Section \ref{sec:cons-sequ-dga-1}) and does
not require the language of consistent sequences of DGAs. However,
additional geometric input is needed in order to define
$\mathcal{A}\mathrm{ug}_+$. For a Legendrian submanifold $\Lambda$
in the jet space $J^1(Q)$, we denote by $\Lambda_n$ the Legendrian
link obtained by taking a generic small perturbation of $n$-copies
of $\Lambda$, which we denote by $\Lambda_1,\hdots, \Lambda_n$,
shifted in the direction of the Reeb vector field by less than the
length of the smallest Reeb chords of $\Lambda$ divided by $n$. Note
that for any chord $c$ of $\Lambda$ and pair $(i,j)$ with
$i,j\in\{1,\hdots,n\}$ there is a corresponding chord $d_{i,j}$ of
$\Lambda_n$ starting on the $i$-th copy of $\Lambda$ and ending on
the $j$-th copy.

The description of the consistent sequence requires  some care in
the choice of the perturbation of the $n$-copy link. Let
$\{f_n\}_{n\in \mathbb{N}}$ be a family of Morse functions on $S^1$
each of which has only two critical points $m_n$ and $M_n$ and
satisfies the following conditions:
\begin{enumerate}
\item for any $n$, there is an inclusion of the oriented intervals $(M_{n+1},m_{n+1})
\subset  (M_{n},m_n)$ (i.e. the critical points are
\textit{nested}).
\item For any $i<j$, $f_i-f_j$ is a Morse function with only two critical points:
one maximum contained in $(M_i,M_j)$ and one minimum in $(m_j,m_i)$.
\end{enumerate}
\begin{Rem}
  The existence of sequences of Morse functions whose critical points
  share similar combinatorics is one technical point that need be
  resolved to extend the definition of $\mathcal{A}\mathrm{ug}_+$ to Legendrian
  submanifolds of high dimensions. Some notion of coherent system of
  perturbations in the spirit of (but not exactly like)
  \cite{Seidel_Fukaya} would allow one to extend this setup and get a
  well defined category with the correct invariance properties.
\end{Rem}

Using this perturbation, any chord of the $i$-copy link becomes a
chords of the $j$-copy link when $\{1,\hdots,i\}$ is a subset of
$\{1,\hdots,j\}$. The Chekanov-Eliashberg DGAs associated to this
sequence of Legendrian links form a consistent sequence of
differential graded algebras where $\mathcal{A}^{(i)}$ is the
Chekanov-Eliashberg DGA of the link $\Lambda_i$. The link grading is
defined by taking the beginning and end of Reeb chords; see Figure
\ref{fig:3-6copy}.

This consistent sequence of differential graded algebra allows us to
define ``plus'' variants of the previous ``minus'' augmentation
categories that we now proceed to outline.

\paragraph{\textbf{The categories $\mathcal{A}\mathrm{ug}_+(\Lambda,R[\pi_1(\Lambda)])$ and $\mathcal{R}\mathrm{ep}_+(\Lambda,n)$.}}
\label{Aug+} The objects are the same as in\\
$\mathcal{A}\mathrm{ug}_-(\Lambda,R[\pi_1(\Lambda)])$ and
$\mathcal{R}\mathrm{ep}(\Lambda,n)$ but the morphisms space are
generated by Reeb chords of $\Lambda$ and critical points of $f_1$.
They can be identified to mixed chords in $\Lambda_2$ and therefore
the morphism spaces correspond to $M^{(2)}$ in the language of
Section \ref{sec:cons-sequ-dga-1}. The operations are defined by
applying the dualisation process of subsection \ref{subsec:case2} to
the maps given by equation \eqref{eq:3}.

Figures \ref{fig:aug-} and \ref{fig:rep-} show again how to compute
such operation: note that in this situation the (non-augmented)
chords $d_2$ and $d_4$ are mixed chords of the $3$-copy link  and
can come from critical points of the functions $f_i-f_j$, the
remaining (augmented) chords $d_1,d_3$ and $d_5$ are pure chords.

\paragraph{\textbf{The category $\mathcal{R}\mathrm{ep}_+(\Lambda,S)$}.}
\label{sec:genrep-1}

This is a similar generalisation of
$\mathcal{R}\mathrm{ep}_-(\Lambda,S)$. Again the contribution of a
pseudoholomorphic disc is as shown in Figure \ref{fig:curve}, with
the understanding that possibly some non-augmented chords are mixed
(and possibly come from Morse critical points).

\begin{Rem}
It follows from the argument in \cite{NRSSZ} that the ``plus''
categories admits a strict unit (for some particular choices of
almost-complex structures), where the unit is given by $m$. Though
this unit is not necessarily closed in the cases
$\mathcal{A}\mathrm{ug}_+(\Lambda,R[\pi_1(\Lambda)])$ and
$\mathcal{R}\mathrm{ep}_+(\Lambda,n)$ (Case II i.e.~Section
\ref{subsec:case2}), it can be seen to be closed in the case of
$\mathcal{R}\mathrm{ep}_+(\Lambda,S)$ (Case I i.e.~Section
\ref{subsec:case1}); see \cite[Theorem 5.5]{Duality_EkholmetAl}. For
instance, in the case of $\mathcal{R}\mathrm{ep}(\Lambda,n)$, it
follows from the description of holomorphic disks having negative
ends asymptotic to the minimum chord $m$ that the boundary of $m$ is
equal to $\sum_{c}\rho(c)^T\cdot m-m\cdot\rho(c)^T$. Due to the
noncommutativity, the latter expression is possibly nonvanishing
whenever $\rho(c)$ is. Note that the construction in Section
\ref{subsec:case1} (Case I) still makes the dual of $m$ a cycle in
this case.
\end{Rem}

\begin{Rem}
It is possible to define $\mathcal{A}\mathrm{ug}_-$ by noting that
if $(b,e)$ is a link grading then $(e,b)$ is also a link grading.
Applying this change to the link grading of $\lambda_n$, the
consistent sequence leading to $\mathcal{A}\mathrm{ug}_+$ gives the
consistent sequence leading to $\mathcal{A}\mathrm{ug}_-$ because
Reeb chords corresponding to critical points if the Morse functions
go in the wrong direction.
\end{Rem}

\subsection{A note about invariance}
The invariance properties satisfied by the constructions carried out
in this paper will not be discussed in detail. Chekanov-Eliashberg
DGA's  are invariant up to so-called ``stable-tame isomorphism''.
From this it is not difficult to see that the set of isomorphism
classes of linearised homologies is invariant, as it was originally
shown in \cite{Chekanov_DGA_Legendrian}. The fact that the
coefficients are noncommuting plays no important role in that proof,
and therefore the same result holds in the current setting as well.

+Similarly the dual complexes constructed in Subsections
\ref{subsec:case1} and \ref{subsec:case2}  satisfy the following
invariance property. Consider the so-called ``bilinearised
co-complexes'' $(M^\vee,\mu_1^{(\varepsilon_1,\varepsilon_2)})$ from
Subsection \ref{subsec:case1} or
$(M,\mu_1^{(\varepsilon_1,\varepsilon_2)})$ from Subsection
\ref{subsec:case2}.  The isomorphism classes of their homologies for
all possible pairs of augmentations $(\varepsilon_1,\varepsilon_2)$
of ${\mathcal A}$ into $A$ is then invariant under stable-tame
isomorphism of the differential graded algebra ${\mathcal A}$. In
fact, since homotopic augmentations (in the sense of DGA morphisms)
induce the same bilinearised (co)complex, this set of isomorphism
classes is even invariant under DGA homotopies (see \cite[Chapter
26]{DGA} for the definition). We refer to \cite[Theorem
2.8]{Bourgeois_Survey} for the proof of a similar statement.

For the invariance properties of the augmentation
$A_\infty$-categories we refer to \cite{augcat}, which handles the
case when $A$ is commutative. It is shown there that the
$A_\infty$-categories associated to two stable-tame isomorphic DGAs
are ``pseudo equivalent'' (see the mentioned article for this
notion). The general case follows similarly. Finally, the invariance
of $\mathcal{A}\mathrm{ug}_+(\Lambda)$ up to quasi-equivalence
follows again from the stable tame isomorphism class of
$\mathcal{A}(\Lambda)$ as shown in \cite{NRSSZ} and can be easily
generalised to noncommutative coefficients.

\bibliographystyle{plain}
\bibliography{Bibliographie_en}

\end{document}